\documentclass{amsart}[11pt]
\usepackage{amsmath, amsfonts, amsthm, amssymb,mathtools}
\usepackage{subfigure}
\usepackage{caption} 
\usepackage{stmaryrd}
\usepackage{verbatim}
\usepackage{hyperref}
\usepackage{ulem}
\usepackage{tikz}
\usetikzlibrary{matrix}
\usepackage{enumerate}
\numberwithin{equation}{section}
\newtheorem{theorem}{Theorem}[section]
\newtheorem{corollary}[theorem]{Corollary}
\newtheorem{lemma}[theorem]{Lemma}
\newtheorem{proposition}[theorem]{Proposition}

\theoremstyle{definition}
\newtheorem{definition}[theorem]{Definition}
\newtheorem{remark}[theorem]{Remark}
\newtheorem{algorithm}[theorem]{Algorithm}
\newtheorem{assumption}[theorem]{Assumption}
\newtheorem{example}[theorem]{Example}

\newcommand{\I}{\mathtt{i}}

\newcommand{\RR}{\mathbb{R}}

\newcommand{\PP}{\mathbb{P}}
\newcommand{\TT}{\mathbb{I}}
\newcommand{\VV}{\mathbb{V}}

\newcommand{\D}{\mathrm{d}}
\newcommand{\Xx}{\mathcal{X}}
\newcommand{\Yy}{\mathcal{Y}}
\newcommand{\TBSS}{\mathcal{TBSS}}
\newcommand{\Bb}{\mathcal{B}}
\newcommand{\Cc}{\mathcal{C}}
\newcommand{\Dd}{\mathcal{D}}
\newcommand{\Ee}{\mathcal{E}}
\newcommand{\Gg}{\mathcal{G}}
\newcommand{\Rk}{\mathfrak{R}}
\newcommand{\Ck}{\mathfrak{C}}

\newcommand{\Rkl}{\mathfrak{R}^{\lambda}}

\newcommand{\Jj}{\mathcal{J}}
\newcommand{\Ll}{\mathcal{L}}
\newcommand{\Lla}{\Ll^{\alpha}}
\newcommand{\Tt}{\mathcal{T}}
\newcommand{\am}{\mathrm{a}}
\newcommand{\bm}{\mathrm{b}}
\newcommand{\cm}{\mathrm{c}}

\newcommand{\E}{\mathrm{e}}

\newcommand{\EE}{\mathbb{E}}

\newcommand{\Nn}{\mathcal{N}}

\newcommand{\Ff}{\mathcal{F}}
\newcommand{\Oo}{\mathcal{O}}
\newcommand{\eps}{\varepsilon}
\newcommand{\ns}{\lfloor ns \rfloor}
\newcommand{\nt}{\lfloor nt \rfloor}
\newcommand{\TTw}{\widetilde{\TT}}
\newcommand{\gr}{\mathfrak{g}}

\newcommand{\sit}{a}

\newcommand{\rrho}{\overline{\rho}}

\DeclareMathOperator*{\esssup}{ess\,sup}

\begin{document}
\title{Functional central limit theorems for rough volatility}
\date{\today}
\author{Blanka Horvath}
\address{Department of Mathematics, University of Oxford}
\email{blanka.horvath@maths.ox.ac.uk}
\author{Antoine Jacquier}
\address{Department of Mathematics, Imperial College London, and the Alan Turing Institute}
\email{a.jacquier@imperial.ac.uk}
\author{Aitor Muguruza}
\address{Kaiju Capital Management and Department of Mathematics, Imperial College London}
\email{aitor.muguruza-gonzalez15@imperial.ac.uk}
\author{Andreas S{\o}jmark}
\address{Department of Statistics, London School of Economics}
\email{a.sojmark@lse.ac.uk}
\thanks{The authors would like to thank Christian Bayer, Peter Friz, Masaaki Fukasawa, Paul Gassiat, 
Jim Gatheral, Mikko Pakkanen and Mathieu Rosenbaum for useful discussions.
BH gratefully acknowledges financial support from the SNSF Early Postdoc.Mobility grant 165248,
and AM is grateful to the Centre for Doctoral Training in Financial Computing \& Analytics for financial support.
Part of this work was carried out while AJ was Visiting Professor in Baruch College, CUNY, 
and AJ further acknowledges financial support from the EPSRC/T032146 grant.
The numerical implementations have been carried out on the collaborative 
\href{www.zanadu.io}{Zanadu} platform, and the code is fully available at
\href{https://github.com/amuguruza/RoughFCLT}{GitHub:RoughFCLT}
}
\maketitle
\begin{abstract}
The non-Markovian nature of rough volatility processes makes Monte Carlo methods challenging
and it is in fact a major challenge to develop fast and accurate simulation algorithms. 
We provide an efficient one for stochastic Volterra processes, based on an extension of Donsker's approximation of Brownian motion to the fractional Brownian case with arbitrary Hurst exponent 
$H \in (0,1)$. 
Some of the most relevant consequences of this `rough Donsker (rDonsker) Theorem' 
are functional weak convergence results in Skorokhod space for discrete approximations of a large class of rough stochastic volatility models. 
This justifies the validity of simple and easy-to-implement Monte-Carlo methods, 
for which we provide detailed numerical recipes. 
We test these against the current benchmark Hybrid scheme~\cite{BLP17} and find remarkable agreement (for a large range of values of~$H$). 
This rDonsker Theorem further provides a weak convergence proof for the Hybrid scheme itself,
and allows to construct binomial trees for rough volatility models, the first available scheme (in the rough volatility context) for early exercise options such as American or Bermudan options.
\end{abstract}
\tableofcontents

\section*{Introduction}

Fractional Brownian motion has a long and famous history in probability, stochastic analysis
and their applications to diverse fields~\cite{Hurst56, HBS65, Kolmogorov40, MV68}.
Recently, it has experienced a new renaissance in the form of fractional volatility models in mathematical finance. 
These were first introduced by Comte and Renault~\cite{CR96}, 
and later studied theoretically by Djehiche and Eddahbi~\cite{DE01}, 
Al\`os, Le\'on and Vives~\cite{ALV07} and Fukasawa~\cite{Fukasawa11}, 
and given financial motivation and data consistency by Gatheral, Jaisson and Rosenbaum~\cite{GJR18} 
and Bayer, Friz and Gatheral~\cite{BFG16}. 
Since then, a vast literature has pushed the analysis in many directions~\cite{BFG+17, BFG+19, BLP16, ER19, FZ17, FTW19, GJR+18, Gulisashvili18, HJL19,  JPS18, NR18}, 
leading to theoretical and practical challenges to understand and implement these models.
One of the main issues, at least from a practical point of view, is on the numerical side:
absence of Markovianity rules out PDE-based schemes, and simulation is the only possibility.
However, classical simulation methods for fractional Brownian motion 
(based on Cholesky decomposition or circulant matrices) 
are notoriously slow, and faster techniques are needed.
The state of the art, so far, is the recent hybrid scheme developed by Bennedsen, Pakkanen and Lunde~\cite{BLP17}, and its turbocharged version~\cite{MP18}.
We rise here to this challenge, and propose an alternative tree-based approach, 
mathematically rooted in an extension of Donsker's theorem to rough volatility.

Donsker~\cite{Donsker51} (and later Lamperti~\cite{Lamperti62}) proved a functional central limit for Brownian motion, 
thereby providing a theoretical justification of its random walk approximation.
Many extensions have been studied in the literature, and we refer the interested reader to~\cite{Dudley99} 
for an overview.
In the fractional case, 
Sottinen~\cite{Sottinen01} and Nieminen~\cite{Nieminen04} constructed--following Donsker's ideas of using iid sequences of random variables-- 
an approximating sequence converging to the fractional Brownian motion, with Hurst parameter $H>1/2$. 
In order to deal with the non-Markovian behaviour of fractional Brownian motion,
Taqqu~\cite{Taqqu75} considered sequences of non-iid random variables, again with the restriction $H>1/2$.
Unfortunately, neither methodologies seem to carry over to the `rough' case $H<1/2$, 
mainly because of the topologies involved. 
The recent development of rough paths theory~\cite{FV10, Lyons98} provided an appropriate framework to extend Donsker's results to processes with sample paths of H\"older regularity strictly smaller than~$1/2$. 
For $H \in(1/3,1/2)$, Bardina, Nourdin, Rovira and Tindel~\cite{BNR+10} 
used rough paths to show that functional central limit theorems (in the spirit of Donsker) apply.
This in particular suggests that the natural topology at work for rough fractional Brownian motion 
is the topology induced by the H\"older norm of the sample paths. 
Indeed, switching the topology from the Skorokhod one used by Donsker to the (stronger) H\"older topology 
is the right setting for rough central limit theorems, as we outline in this paper. 
Recent results~\cite{BP12, Parczewski17, Parczewski14} provide convergence for (geometric) fractional Brownian motions with general $H \in (0,1)$ using Wick calculus,
assuming that the approximating sequences are Bernoulli random variables.
We extend this (Theorem~\ref{thm:MainThm}) to a universal functional central limit theorem, 
involving general (discrete or continuous) random variables as approximating sequences, 
only requiring finiteness of moments.

We consider a general class of continuous processes with any H\"older regularity, 
including fractional Brownian motion with $H\in (0,1)$,
truncated Brownian semi-stationary processes, Gaussian Volterra processes,
as well as rough volatility models recently proposed in the financial literature.
The fundamental novelty here is an approximating sequence capable 
of simultaneously keeping track of the approximated rough volatility process 
(fractional Brownian motion,  Brownian semistationary process, or any continuous path functional thereof) 
and of the underlying Brownian motion. 
This is crucial in order to take into account the correlation of the two processes,
the so-called leverage effect in financial modelling. 
While approximations of two-dimensional (correlated) semimartingales are well understood 
in the standard case, the rough case is so far an open problem.
Our analysis easily generalises beyond Brownian drivers to more general semimartingales, 
emphasising that the subtle, yet essential, difficulties lie in the passage from the semimartingale setup to the rough case. 
This is the first Monte-Carlo method available in the literature, 
specifically tailored to two-dimensional rough systems,
based on an approximating sequence for which 
we prove a Donsker-Lamperti-type functional central limit theorem (FCLT).
This further provides a pathwise justification of the hybrid scheme 
by Bennedsen, Lunde and Pakkanen~\cite{BLP17}, 
and to develop tree-based schemes, opening the doors to pricing early-exercise options such as American options.
In Section~\ref{sec:WeakConv}, we present the class of models we are considering and state our main results. The proof of the main theorem is developed in Section~\ref{sec:AllConvs} in several steps.
We reserve Section~\ref{sec:Applications} to applications of the main result, namely weak convergence of the hybrid scheme, binomial trees as well as numerical examples. We present simple numerical recipes, 
providing a pedestrian alternative to the advanced hybrid schemes in~\cite{BLP17,MP18},
and develop a simple Monte-Carlo with low implementation complexity,
for which we provide comparison charts against~\cite{BLP17} in terms of accuracy 
and against~\cite{MP18} in terms of speed.
Reminders on Riemann-Liouville operators and additional technical proofs are postponed to the appendix.
\vspace{0.3cm}

\textbf{Notations}:
On the interval $\TT:=[0,1]$, $\Cc(\TT)$ and $\Cc^\alpha(\TT)$ 
denote the spaces of continuous and $\alpha$-H\"older continuous functions on~$\TT$ with H\"older regularity $\alpha\in(0,1)$;
$\Cc^1(\TT):=\{f:\TT\to\mathbb{R}: \text{$f'$ exists and is continuous on } \TT\}$ and 
$\Cc^{1}_{+}(\TT):=\{f:\TT\to\mathbb{R_+}: \text{$f'$ exists and is continuous on } \TT\}
$. Both definitions imply bounded first order derivatives on~$\TT$. We use $C,\widetilde{C},\widehat{C}, C_1, C_2,\overline{C},\underline{C}$ as strictly positive real constants which may change from line to line,
the exact values of which do not matter.

\section{Weak convergence of rough volatility models}\label{sec:WeakConv}
Donsker's invariance principle~\cite{Donsker51}
(also termed `functional central limit theorem') ensures the weak convergence of an approximating sequence to a Brownian motion in the Skorokhod space.
As opposed to the central limit theorem, Donsker's theorem is a pathwise statement
which ensures that convergence takes place for all times. 
This result is particularly important for Monte-Carlo methods, 
which aim to approximate pathwise functionals of a given process 
(an essential requirement to price path-dependent financial securities for example).
We prove here a version of Donsker's result, not only in the Skorokhod topology, but also
in the stronger H\"older topology, for a general class of continuous stochastic processes.

\subsection{H\"older spaces and fractional operators}
For $\beta\in (0,1]$, the $\beta$-H\"older space $\Cc^\beta(\TT)$,
with the norm 
$$
\| f\|_{\beta}
 := |f|_{\beta} + \| f\|_\infty
  = \sup_{\substack{t,s\in\TT \\ t\neq s}}\frac{|f(t)-f(s)|}{|t-s|^{\beta}} + \max_{t\in\TT} |f(t)|,
$$
is a non-separable Banach space~\cite[Chapter 3]{Krylov96}.
In the spirit of Riemann-Liouville fractional operators 
recalled in Appendix~\ref{sec:Riemann-Liouville operators}, 
we introduce Generalised Fractional Operators (GFO).
For $\lambda \in (0,1)$, define the intervals
$$
\Rkl := (-\lambda, 1-\lambda),
\qquad
\Rkl_+ :=\Rkl\cap (0,1),
\qquad
\Rkl_- :=\Rkl\cap (-1,0),
$$
and the space 
$\Lla := \{ g\in\Cc^{2}((0,1]) : \big|\frac{g(u)}{u^{\alpha}}\big|, \big|\frac{g^\prime(u)}{u^{\alpha-1}}\big|
\text{ and } \big|\frac{g^{\prime\prime}(u)}{u^{\alpha-2}}\big| \text{ bounded}\}$,
for $\alpha\in\Rkl$.
\begin{definition}\label{def: GFO}
For any $\lambda\in(0,1)$ and $\alpha\in \Rkl$, 
the GFO associated to $g\in \Lla$ is defined on~$\Cc^\lambda(\TT)$ as
\begin{equation}\label{eq:GFO}
(\Gg^\alpha f)(t) := 
\left\{
\begin{array}{ll}
\displaystyle \int_{0}^{t}(f(s)-f(0))\frac{\D}{\D t}g(t-s)\D s, & \text{if } \alpha \in (0,1-\lambda),\\
\displaystyle \frac{\D}{\D t}\int_{0}^{t}(f(s)-f(0))g(t-s)\D s, & \text{if } \alpha \in (-\lambda, 0).
\end{array}
\right.
\end{equation}
\end{definition}
We shall further use the notation $G(t):=\int_{0}^{t}g(u)\D u$, for any $t\in\TT$.
Of particular interest in mathematical finance are the following kernels:
\begin{equation}\label{ex:GFOs}
\begin{array}{lll}
\text{Riemann-Liouville: }
& g(u) = u^{\alpha}, & \text{for } \alpha \in (-1, 1);\\
\text{Gamma fractional: }
& g(u) = u^{\alpha}\E^{\beta u}, & \text{for }\alpha \in (-1, 1), \beta<0;\\
\text{Power-law: }
& g(u) = u^{\alpha}(1+u)^{\beta-\alpha}, & \text{for } \alpha \in (-1, 1), \beta<-1.
\end{array}
\end{equation}
The next result generalises the classical mapping properties of Riemann-Liouville fractional operators 
first proved by Hardy and Littlewood~\cite{HL28}, 
and will be of fundamental importance in the rest of our analysis.

\begin{proposition}\label{prop:GContinuous}
For any $\lambda\in(0,1)$ and $\alpha\in \Rkl$, the operator 
$\Gg^\alpha$ is continuous
from $\Cc^{\lambda}(\TT)$ to $\Cc^{\lambda+\alpha}(\TT)$. 
\end{proposition}

The proof can be found in Appendix \ref{sect:app_C}. We note that the result is analogous to the classical Schauder estimates, phrased in terms of convolution with a suitable regularising kernel, as e.g.~treated in~\cite[Theorem 14.17]{FH20} and~\cite[Theorem~2.13 and Lemma 2.9]{Zambotti23}, but in settings that are slightly different from ours.

We develop here an approximation scheme for the following system, 
generalising the concept of rough volatility 
introduced in~\cite{ALV07, Fukasawa11, GJR18} in the context of mathematical finance,
where the process~$X$ represents the dynamics of the logarithm of a stock price process:
\begin{equation}\label{eq:System}
\begin{array}{rll}
\D X_t & = \displaystyle -\frac{1}{2}V_t \D t + \sqrt{V_t}\D B_t, & X_0 = 0,\\
V_t & = \displaystyle \Phi\left(\Gg^\alpha Y\right)(t),
 & \displaystyle 
\end{array}
\end{equation}
with $\alpha\in(-\frac{1}{2},\frac{1}{2})$, 
 and~$Y$ the (strong) solution to the stochastic differential equation
\begin{equation}\label{eq:diffusion}
\D Y_t = b(Y_t)\D t + \sit(Y_t)\D W_t, \quad Y_0\in\Dd_Y,
\end{equation}
where~$\Dd_Y$ denotes the state space of~$Y$, usually~$\RR$ or~$\RR_+$.
The two Brownian motions~$B$ and~$W$, defined on a common filtered probability space
$(\Omega, \Ff, (\Ff_t)_{t\in\TT}, \PP)$, are correlated by the parameter $\rho \in [-1,1]$, 
and we let the operator $\Phi$ be such that, for all $
\gamma \in (0,1)$, we have $\Phi:\mathcal{C}^{\gamma}(\TT)\to \mathcal{C}^{\gamma}_+(\TT)$ with $\Phi$ continuous from $\left(\Cc^\gamma(\TT),\|\cdot\|_\gamma\right)$ to itself.
This in particular implies that whenever $Y \in C^\lambda(\TT)$ then $V \in C_+^{\alpha+\lambda}(\TT)$, i.e., $V$ is non-negative and belongs to $C^{\alpha+\lambda}(\TT)$. 
As an example, one can consider a so-called Nemyckij operator $\Phi(f):=\phi \circ f$, given by composition with some $\phi:\mathbb{R} \rightarrow \mathbb{R}_+$, in which case Dr\'abek~\cite{Drabek75} has shown that the operator $\Phi$ is continuous from 
$\left(\Cc^\gamma(\TT),\|\cdot\|_\gamma\right)$ to $\left(\Cc^\gamma(\TT),\|\cdot\|_\gamma\right)$, for all $\gamma\in(0,1)$, 
if and only if $\phi \in\Cc^1(\mathbb{R})$.
It remains to formulate a precise definition for~$\Gg^\alpha W$ (Proposition~\ref{prop: mapping BM}) 
and for~$\Gg^\alpha Y$ (Corollary~\ref{cor:GalphaY}) 
to fully specify the system~\eqref{eq:System} and clarify the existence of solutions. 
\begin{assumption}\label{assu:CoeffSDE}
There exist $C_b, C_{\sit}>0$ such that, for all $y \in\Dd_Y$,
$$
|b(y)| \leq C_b (1+|y|)
\qquad\text{and}\qquad
|\sit(y)| \leq C_{\sit} (1+|y|),
$$ where $a$ and $b$ are continuous functions such that there is a unique strong solution to~\eqref{eq:diffusion}.
\end{assumption}
\noindent Existence of solutions to~\eqref{eq:diffusion} 
along with Assumption~\ref{assu:CoeffSDE} need to be checked on a case by case basis beyond standard Lipschitz and linear growth conditions (we provide a showcase of models 
in Examples~\ref{ex:1}-\ref{ex:3} below satisfying existence and pathwise uniqueness conditions). General conditions for stochastic invariance can be found in~\cite{AD03, BJ02, DF07} for diffusions, and in~\cite{ALP19} for affine Volterra processes.
Not only is the solution to~\eqref{eq:diffusion} continuous, but $(\frac{1}{2}-\eps)$-H\"older continuous for any~$\eps \in (0,\frac{1}{2})$ as a consequence of the Kolmogorov-\u{C}entsov theorem~\cite{Centsov56}.
Existence and precise meaning of~$\Gg^\alpha Y$ is delicate, 
and is treated below.

\subsection{Examples}
Before constructing our approximation scheme, let us discuss a few examples of processes within our framework.
As a first useful application, these generalised fractional operators render a (continuous) mapping between 
a standard Brownian motion and its fractional counterpart:
\begin{proposition}\label{prop: mapping BM}
For any $\alpha \in \Rk^{1/2}$, the equality $(\Gg^\alpha W)(t)=\int_0^t g(t-s)\D W_s$
holds almost surely for all $t\in\TT$.
\end{proposition}

\begin{proof}
Since the paths of Brownian motion are $(\frac{1}{2}-\eps)$-H\"older continuous for any $\eps\in(0,\frac{1}{2})$, 
existence (and continuity) of~$\Gg^\alpha W$ is guaranteed for all $\alpha \in \Rk^{1/2}$.
When $\alpha\in\Rk^{1/2}_+$, the kernel is smooth and square integrable, so that It\^o's product rule 
yields (since $g(0)=0$)
\begin{align*}
(\Gg^\alpha W)(t) &= \int_0^t \frac{\D}{\D t}g(t-s)(W(s)-W(0))\D s =g(t)(W(0)-W(0))-g(0)\left(W(t)-W(0)\right)+ \int_0^t g(t-s)\D W_s,\\
&=\int_0^t g(t-s)\D W_s,
\end{align*}
and the claim holds.
For $\alpha\in\Rk^{1/2}_-$, and any $\eps>0$, introduce the operator
$$
\left(\Gg^{1+\alpha}_{\eps}f\right)(t) := \int_0^{t-\eps}g(t-s)(f(s)-f(0))\D s,
\qquad\text{for all }t\in\TT,
$$
which satisfies 
$\frac{\D}{\D t}\lim_{\eps\downarrow 0}\left(\Gg^{1+\alpha}_{\eps}f\right)(t)
 = \left(\Gg^{\alpha}f\right)(t)$ pointwise.
Now, for any $t\in\TT$, almost surely,
\begin{align}\label{eq:alpha<0}
\frac{\D}{\D t}  \left(\Gg^{1+\alpha}_{\eps}W\right)(t) &= g(\eps)\left(W(t-\eps)-W(0)\right) -g(t)\left(W(0)-W(0)\right)+ \int_0^{t-\eps}\frac{\D}{\D t}g(t-s)W(s) \D s \nonumber\\
&=g(\eps)W(0)+\int_0^{t-\eps}g(t-s)\D W_s.
\end{align}
Then, as $\eps$ tends to zero, the right-hand side of~\eqref{eq:alpha<0} tends to
$\int_0^{t}g(t-s)\D W_s$,
and furthermore, the convergence is uniform.
On the other hand, the equalities
\begin{align*}
(\Gg_0^{1+\alpha}W)(t)-(\Gg_0^{1+\alpha}W)(0)
 & = \lim_{\eps\downarrow 0}
 \Big[(\Gg^{1+\alpha}_{\eps}W)(t)-(\Gg^{1+\alpha}_{\eps}W)(0)\Big]
 = \lim_{\eps\downarrow 0}\int_{0}^{t}\left(\frac{\D}{\D s}\Gg^{1+\alpha}_{\eps}W\right)(s)\D s\\
 & = \int_{0}^{t}\lim_{\eps\downarrow 0}\left(\frac{\D}{\D s}\Gg^{1+\alpha}_{\eps}W\right)(s)\D s
  = \int_{0}^{t}\left(\int_0^{s}g(s-u)\D W_u\right)\D s,
\end{align*}
hold since convergence is uniform on compacts, 
and the fundamental theorem of calculus concludes the proof.
\end{proof}

Modulo a constant multiplicative factor~$C_\alpha$, 
the (left) fractional Riemann-Liouville operator (Appendix~\ref{sec:Riemann-Liouville operators})
is identical to the GFO in~\eqref{ex:GFOs},
so that the Riemann-Liouville (or Type-II) fractional Brownian motion can be written as $C_\alpha \Gg^\alpha W$.
Proposition~\ref{prop:GContinuous} then implies that the Riemann-Liouville operator 
is continuous from~$\Cc^{1/2}(\TT)$ to~$\Cc^{1/2+\alpha}(\TT)$ for $\alpha\in\Rk^{1/2}$.
Each kernel in~\eqref{ex:GFOs} gives rise to processes proposed 
by Barndorff-Nielsen and Schmiegel~\cite{BS07} for turbulence and financial modelling.
\begin{example}\label{ex:1}
The rough Bergomi model introduced 
by Bayer, Friz and Gatheral~\cite{BFG16} reads
$$
V_t =  \xi_0(t)\Ee\left(2\nu C_{H}\int_0^t(t-s)^{\alpha}\D W_s\right),
$$
with $V_0, \nu,\xi_0(\cdot)>0$, $\alpha\in\Rk^{1/2}$ 
and $\Ee(\cdot)$ is the Wick stochastic exponential.
This corresponds exactly to~\eqref{eq:System} with
$g(u) \equiv u^\alpha$, $Y = W$ and
$$
\Phi(\varphi)(t) := \xi_0(t) \exp\left(2\nu C_{H}\varphi(t)\right)\exp\left\{-2\nu^2C_{H}^2\int_{0}^{t}(t-s)^{2\alpha}\D s\right\}.
$$
\end{example}

\begin{example}\label{ex:4}
A truncated Brownian semistationary ($\TBSS$) process is defined as
$\int_{0}^{t} g(t-s)\sigma(s)\D W_s$, for $t\in\TT$,
where~$\sigma$ is $(\Ff_t)_{t\in\TT}$-predictable with locally bounded trajectories and finite second moments, and $g: \TT\setminus\{0\}\to\TT$ is Borel measurable and square integrable. 
If $\sigma\in\Cc^1(\TT)$, this class falls within the GFO framework.
\end{example}

\begin{example}\label{ex:2}
Bennedsen, Lunde and Pakkanen~\cite{BLP16} considered adding a Gamma kernel to the volatility process, 
which yields the Truncated Brownian semi-stationary (Bergomi-type) model:
$$
V_t =  \xi_0(t)\Ee\left(2\nu C_{H}\int_0^t(t-s)^{\alpha}\E^{-\beta(t-s)}\D W_s\right),
$$
with $\beta>0$, $\alpha\in\Rk^{1/2}$.
This corresponds to~\eqref{eq:System} with $Y=W$,
Gamma fractional kernel $g(u) \equiv u^{\alpha} \E^{-\beta u}$ in~\eqref{ex:GFOs},
$$
\Phi(\varphi)(t)
 := \xi_0(t) \exp\left(2\nu C_{H}\varphi(t)\right)\exp\left\{-2\nu^2C_{H}^2\int_{0}^{t}(t-s)^{2\alpha} \E^{-2\beta (t-s)}\D s\right\}.
$$
\end{example}

\begin{example}\label{ex:3}
The rough Heston model introduced by Guennoun, Jacquier, Roome and Shi~\cite{GJR+18} reads
\begin{equation*}
\begin{array}{rll}
Y_t & \displaystyle=  Y_0 + \int_0^t\kappa(\theta-Y_s)\D t+\int_0^t\xi\sqrt{Y_s}\D W_s,\\
 V_t & \displaystyle =\eta+\int_0^t(t-s)^\alpha \D Y_s,
\end{array}
\end{equation*}
with $Y_0,\kappa,\xi,\theta>0$, $2\kappa\theta>\xi^2$ and $\eta>0$, $\alpha\in\Rk^{1/2}$.
This corresponds exactly to~\eqref{eq:System} with
$g(u) \equiv u^{\alpha}$, $\Phi(\varphi)(t) := \eta + \varphi(t)$,
and the coefficients of~\eqref{eq:diffusion} read $b(y) \equiv \kappa(\theta-y)$ and $a(y) \equiv \xi\sqrt{y}$.
This model is markedly different from the rough Heston introduced by El Euch and Rosenbaum~\cite{ER19} (for which the characteristic function is known in semi-closed form). 
Unfortunately, this second version is out of the scope of our invariance principle. 
\end{example}

\subsection{The approximation scheme}
We now move on to the core of the project, namely an approximation scheme for the system~\eqref{eq:System}-\eqref{eq:diffusion}.
The basic ingredient to construct approximating sequences will be suitable families of iid random variables
which satisfy the following assumption:
\begin{assumption}\label{assu:xi}
The family $(\xi_i)_{i\geq 1}$ forms an iid sequence of centered random variables 
with finite moments of all orders and $\EE[\xi_1^2] = \sigma^2>0$.
\end{assumption}
Given $(\zeta_i)_{i\geq 1}$ satisfying Assumption \ref{assu:xi}, Lamperti's~\cite{Lamperti62} generalisation of Donsker's~\cite{Donsker51} invariance principle tells us that a Brownian motion~$W$ can be approximated weakly in H{\"o}lder space (for details, see Theorem \ref{thm: Donsker}) by processes of the form
\begin{equation}\label{eq:Donsker}
W_n(t) := \frac{1}{\sigma\sqrt{n}}\sum_{k=1}^{\nt}\zeta_k
 + \frac{nt-\nt}{\sigma\sqrt{n}}\zeta_{\nt+1},
\end{equation}
defined pathwise for any $\omega\in\Omega$, $n\geq 1$, and $t\in\TT$. As we explain in Section~\ref{subsect:Ito_diffusions_Holder}, a similar construction holds to weakly approximate the process~$Y$ from~\eqref{eq:diffusion} in H{\"o}lder space:
\begin{equation}\label{eq:Donskerdiffusion}
Y_n(t)
 := Y_n(0) + 
 \frac{1}{n}\sum_{k=1}^{\nt}b\left(Y_n^{k-1}\right)
 +\frac{nt-\nt}{n}b\left(Y_n^{\nt}\right)
 + \frac{1}{\sigma\sqrt{n}}\sum_{k=1}^{\nt}\sit\left(Y_n^{k-1}\right)\zeta_k
 + \frac{nt-\nt}{\sigma\sqrt{n}}\sit\left(Y_n^{\nt}\right)\zeta_{\nt+1},
\end{equation}
where $Y_n^k:=Y_n(t_k)$ and $\Tt_n:=\{t_k = \frac{k}{n}\}_{k=0,...,n}$. 
Here the $\zeta_i$'s correspond to the innovations of the Brownian motion~$W$ in~\eqref{eq:diffusion}. Similarly, we shall use $\xi_i$ when referring to the innovations of the Brownian~$B$ from~\eqref{eq:System} which enter into the approximations of the log stock price in~\eqref{eq:Donsker log stock} below. 
Throughout the paper, we assume that the innovations $\{\xi_i\}_{i=1}^{\nt}$ and $\{\zeta_i\}_{i=1}^{\nt}$ come from two sequences $(\xi_i)_{i\geq1}$ and $(\zeta_i)_{i\geq1}$ satisfying Assumption~\ref{assu:xi} such that $((\xi_i,\zeta_i))_{i\geq1}$ is i.i.d.~with $\text{corr}(\xi_i,\zeta_i)=\rho$ for all $i\geq 1$.
Naturally, the approximations in~\eqref{eq:Donskerdiffusion} and in~\eqref{eq:Donsker log stock} below should be understood pathwise, but we omit the~$\omega$-dependence in the notations for clarity.

Regarding the approximation scheme for the process~$X$, given by~\eqref{eq:System}, we follow a typical route in weak convergence analysis~\cite{Billingsley68,Ethier_Kurtz}  and establish convergence in the Skorokhod space $\left(\Dd(\TT),d_{\Dd}\right)$. Here $\Dd(\TT)=\Dd(\TT,\mathbb{R})$ denotes the space of $\mathbb{R}$-valued c\`adl\`ag processes on $\TT$ and $d_{\Dd}$ denotes a metric inducing the Skorokhod topology. To approximate $X$ in this space, we shall then consider the following process:
\begin{equation}\label{eq:Donsker log stock}
\begin{array}{ll}
X_n(t)
 &:=\; \displaystyle-\frac{1}{2n}\sum_{k=1}^{\nt}\Phi\left(\Gg^\alpha Y_n\right)(t_{k-1})
  + \frac{1}{\sigma \sqrt{n}}\sum_{k=1}^{\nt}\sqrt{\Phi\left(\Gg^\alpha Y_n\right)(t_{k-1})}\xi_{k}.
\end{array}
\end{equation}
Analogously to~\eqref{eq:Donskerdiffusion}, one could view these as continuous processes via linear interpolation, but we note that the interpolating term would decay to zero by Chebyshev's inequality.
The following result, proved in Section~\ref{sec:ProofMainThm}, confirms the functional convergence of the approximating sequence~$(X_n)_{n\geq 1}$. 
\begin{theorem}\label{thm:MainThm}
The sequence $(X_n)_{n\geq 1}$ converges weakly to~$X$ in $(\Dd(\TT),d_{\Dd})$.
\end{theorem}
The construction of the proof allows to extend the convergence to the case where~$Y$ is a $d$-dimensional diffusion without additional work.
The proof of the theorem requires a certain number of steps:
we start with the convergence of the approximations~$(Y_n)$, in some H\"older space,
which we then translate into convergence of the sequence~$(\Phi(\Gg^\alpha Y_n))$, by suitable continuity properties of the operations $\Gg^\alpha$ and $\Phi$, before finally deducing also the convergence of the corresponding stochastic integrals for the approximations of~\eqref{eq:System}. These steps are carried out in Sections \ref{subsect:Ito_diffusions_Holder}, \ref{sect:invariance_principles}, and \ref{sec:ProofMainThm} below.

\section{Functional Central limit theorems for a family of H\"older continuous processes}\label{sec:AllConvs}

\subsection{Weak convergence of Brownian motion in H\"older spaces}
Donsker's classical convergence result was proven under the Skorokhod topology. 
We concentrate here on convergence in the H\"older topology, due to  Lamperti~\cite{Lamperti62a}.
The standard convergence result for Brownian motion can be stated as follows:
\begin{theorem}\label{thm: Donsker}
For $\lambda<\frac{1}{2}$, the sequence $(W_n)$ in~\eqref{eq:Donsker} 
converges weakly to a Brownian motion in 
$\left(\Cc^\lambda(\TT),\|\cdot\|_\lambda\right)$.
\end{theorem}
The proof relies on finite-dimensional convergence and tightness of the approximating sequence. 
Not surprisingly, the tightness criterion~\cite{Billingsley68} in the Skorokhod space $\Dd(\TT)$ 
and in a H\"older space are different. 
In fact, the tightness criterion in H\"older space is strictly related to Kolmogorov-\u{C}entsov's continuity~\cite{Centsov56}.
Note in passing that the approximating sequence~\eqref{eq:Donsker} is piecewise differentiable in time
for each $n\geq 1$ even though its limit is obviously not. The proof of Theorem~\ref{thm: Donsker} follows from Theorem~\ref{thm: tailxi} under Assumption~\ref{assu:xi}.
\begin{theorem}[Sufficient conditions for weak convergence in H\"older spaces (Ra\v{c}kauskas-Suquet~\cite{RS04})]\label{thm: weak convergence Holder}
Let $Z\in\Cc^\lambda(\TT)$ and $(Z_n)_{n\geq 1}$ 
an approximating sequence in the sense that,
for any sequence~$(\tau_k)_k$ in~$\TT$,
$(Z_n(\tau_k))_k$ converges in distribution to 
$ (Z(\tau_k))_k$ as~$n$ tends to infinity. Assume further that we have
\begin{equation}\label{eq:Tightness}
\EE\left[|Z_n(t) - Z_n(s)|^\gamma\right] \leq C|t-s|^{1+\beta}
\end{equation}
for all $n\geq 1$, $t,s\in\TT$, for some $C,\gamma,\beta >0$ with $\frac{\beta} {\gamma} \leq \lambda$.
Then $(Z_n)_{n\geq 1}$ converges weakly to~$Z$ in $\Cc^\mu(\TT)$ for $\mu<\frac{\beta}{\gamma} \leq \lambda$.
\end{theorem}
The proof of this theorem relies on results of Ra\v{c}kauskas and Suquet~\cite{RS04}, 
who prove the convergence in the H\"older space $C_0^{\lambda}(\TT)$ endowed with the norm 
$\|f\|^0_\lambda:=|f|_\lambda+|f(0)|$, 
for all functions that satisfy 
$$
\lim_{\delta\downarrow 0} \sup_{\substack{0<t-s<\delta\\
t,s\in\TT}}\frac{|f(t)-f(s)|}{(t-s)^{\gamma}}=0.
$$
From here the proof of Theorem~\ref{thm: weak convergence Holder} is a straightforward consequence, since $\left(C_0^{\lambda}(\TT),\|\cdot\|^0_\lambda\right)$ is a separable closed subspace of $\left(\Cc^{\lambda}(\TT),\|\cdot\|_\lambda\right)$ 
(see~\cite{Hamadouche00, RS04} for details), 
and one can then use the simple tightness criterion introduced above to conclude.
Moreover, as the identity map from $C_0^{\lambda}(\TT)$ into $\Cc^{\lambda}(\TT)$ is continuous,
weak convergence in the former implies weak convergence in the latter.
To conclude our review of weak convergence in H\"older spaces, 
the following theorem, due to Ra\v{c}kauskas and Suquet~\cite{RS04}
provides necessary and sufficient conditions ensuring convergence in H\"older space:
\begin{theorem}[Ra\v{c}kauskas-Suquet~\cite{RS04}]\label{thm: tailxi}
For any $\lambda \in (0, \frac{1}{2})$, the sequence $(W_n)_{n\geq 1}$ in~\eqref{eq:Donsker}
converges weakly to a Brownian motion in~$\Cc^\lambda(\TT)$  if and only if
$\EE[\xi_1] = 0$
and $\lim\limits_{t\uparrow \infty}t^{\frac{1}{1-2\lambda}}\mathbb{P}(|\xi_1|\geq t) = 0$.
\end{theorem} 
Assumption~\ref{assu:xi} ensures the conditions in Theorem~\ref{thm: tailxi}.  
The following statement allows us to apply Theorem~\ref{thm: weak convergence Holder} 
on~$\TT$ and extend the H\"older convergence result via linear interpolation to a continuous sequence. 
\begin{theorem}\label{thm:HolderConvergencePartition}
Let $Z\in\Cc^\lambda(\TT)$ and $(Z_n)_{n\geq 1}$ 
an approximation sequence such that finite-dimensional convergence holds, i.e. 
$Z_n(t)$ converges in distribution to 
$Z(t)$ for $t\in\TT$ as~$n$ tends to infinity. 
Moreover, if
\begin{equation}\label{eq:IneqZ}
\EE\left[\left|Z_n(t_i)-Z_n(t_j)\right|^{\gamma}\right]\leq C \left|t_i-t_j\right|^{1+\beta},
\end{equation}
for any $t_i, t_j \in \Tt_n$ and some $\beta, \gamma, C>0$ with $\tfrac{\beta}{\gamma}\leq \lambda$ and $\gamma \geq 1+\beta$,
then the linear interpolating sequence 
$$
\overline{Z}_n(t):=Z_n\left(\frac{\nt}{n}\right) +(nt-\nt)\left(Z_n\left(\frac{\nt+1}{n}\right) -Z_n\left(\frac{\nt}{n}\right) \right)
$$
satisfies \eqref{eq:Tightness}. In particular, $\overline{Z}_n$ then converges weakly to~$Z$ in $\Cc^\mu(\TT)$ for $\mu<\frac{\beta}{\gamma} \leq \lambda$.
\end{theorem}
\begin{proof}
For any $t, s \in \TT$, we can write, letting $Z_n^k := Z_n(t_k)$ and $\overline{Z}_n^k := \overline{Z}_n(t_k)$,
\begin{align*}
\EE\left[|\overline{Z}_n(t)-\overline{Z}_n(s)|^{\gamma}\right]
 &= \EE\left[\left|Z_n^{\nt} + (nt - \nt)\left(Z_n^{\nt+1} - Z_n^{\nt}\right)
 - Z_n^{\ns} - (ns - \ns)\left(Z_n^{\ns+1} - Z_n^{\ns}\right)\right|^{\gamma}\right]\\
& \leq 3^{\gamma-1} \EE\left[\left|Z_n^{\nt} - Z_n^{\ns}\right|^{\gamma}
 + (nt-\nt)^{\gamma}\left|Z_n^{\nt+1} - Z_n^{\nt}\right|^{\gamma}
 + (ns-\ns)^{\gamma}\left|Z_n^{\ns+1} - Z_n^{\ns}\right|^\gamma\right]\\
& \leq C \left(\left(\frac{\nt-\ns}{n}\right)^{1+\beta}
 + \frac{(nt-\nt)^{\gamma}}{n^{1+\beta}}
 + \frac{(ns-\ns)^{\gamma}}{n^{1+\beta}}\right)
\leq C (t-s)^{1+\beta},
\end{align*}
where we used~\eqref{eq:IneqZ} and the fact that $\frac{\nt-\ns}{n}\leq 2(t-s)$, 
$nt -\nt\leq 1$ for $t\geq 0$ and $\frac{1}{n}\leq (t-s)$. 

Finally, it is left to prove the case  $\frac{1}{n}> (t-s)$. There are two possible scenarios here:
\begin{itemize}
\item If $\nt=\ns$, then, using $\gamma \geq 1+\beta$, we have
$$
\EE[|\overline{Z}_n(t)-\overline{Z}_n(s)|^{\gamma}]
 = \EE\left[\left|( nt-ns)\left(Z_n^{\nt+1} - Z_n^{\nt}\right)\right|^{\gamma}\right]
\leq \frac{C|t-s|^{\gamma}}{ n^{1+\beta-\gamma}} \leq C (t-s)^{1+\beta}
$$
\item If $\nt\neq \ns$, 
then either  $\nt+1= \ns$ or $\nt= \ns+1$. 
Without loss of generality consider the second case. Then
\begin{align*}
\EE\left[|\overline{Z}_n(t)-\overline{Z}_n(s)|^{\gamma}\right]
 & = \EE\left[\left|\overline{Z}_n(t)-Z_n^{\nt}+Z_n^{\nt} - \overline{Z}_n(s)\right|^{\gamma}\right]
\leq 2^{\gamma-1} \EE\left[\left|\overline{Z}_n(t) - Z_n^{\nt}\right|^{\gamma}+\left|Z_n^{\nt}-\overline{Z}_n(s)\right|^{\gamma}\right]\\
&\leq C \left((t-s)^{1+\beta}+\EE\left[\left|(\nt-ns)\left(Z_n^{\nt} -Z_n^{\nt-1}\right)\right|^{\gamma}\right]\right),
\end{align*} 
and the result follows as before since 
$t-\frac{\nt}{n}<|t-s|$ and $|s-\frac{\nt}{n}|\leq |t-s|.$
\end{itemize}
\end{proof}

\subsection{Weak convergence of It\^o diffusions in H\"older spaces}\label{subsect:Ito_diffusions_Holder}
The first important step in our analysis is to extend Donsker-Lamperti's weak convergence 
from Brownian motion to the It\^o diffusion~$Y$ in~\eqref{eq:diffusion}.
\begin{theorem}\label{thm: Donsker diffusion}
The sequence $(Y_n)_{n\geq 1}$ in~\eqref{eq:Donskerdiffusion} 
converges weakly to~$Y$ in~\eqref{eq:diffusion} 
in $\left(\Cc^\lambda(\TT),\|\cdot\|_\lambda\right)$ for all $\lambda<\frac{1}{2}$, 
\end{theorem}

\begin{proof}
Finite-dimensional convergence is a classical result by Kushner~\cite{Kushner74}, 
so only tightness needs to be checked. 
In particular, using Theorem~\ref{thm:HolderConvergencePartition} 
we need only consider the partition~$\Tt_n$. Thus, we get
\begin{align*}
\EE\left[|Y_n^{j}-Y_n^i|^{2p}\right]
 &=  \EE\left[\left| \sum_{k=i+1}^j \frac{1}{n}b\left(Y_n^{k-1}\right)+\frac{1}{\sigma\sqrt{n}}a\left(Y_n^{k-1}\right)\zeta_{k}\right|^{2p}\right]\\
 &  \leq 2^{2p-1}\left\lbrace \EE\left[\left| \sum_{k=i+1}^j \frac{1}{n}b\left(Y_n^{k-1}\right)\right|^{2p}\right]+\EE\left[\left|\sum_{k=i+1}^j\frac{a\left(Y_n^{k-1}\right)\zeta_{k}}{\sigma\sqrt{n}}\right|^{2p}\right]\right\rbrace \\
  & \leq 2^{2p-1} \left\lbrace \EE\left[\left| \sum_{k=i+1}^j \frac{1}{n}b\left(Y_n^{k-1}\right)\right|^{2p}\right]+C(p)\EE\left[\left|\sum_{k=i+1}^j\frac{a\left(Y_n^{k-1}\right)^2 \zeta_k^2}{\sigma^2n}\right|^{p}\right]\right\rbrace\\
  & \leq 2^{2p-1} \left\lbrace \frac{(j-i)^{2p-1}}{n^{2p}} \sum_{k=i+1}^j\EE\left[\left|b\left(Y_n^{k-1}\right)\right|^{2p}\right]+\frac{(j-i)^{p-1}}{n^p}C(p)\frac{\EE[\zeta_1^{2p}]}{\sigma^{2p}}\sum_{k=i+1}^j\EE\left[a\left(Y_n^{k-1}\right)^{2p}\right]\right\rbrace\\
  &\leq 2^{2p-1} \frac{(j-i)^{p-1}}{n^{p}}\sum_{k=i+1}^j \left(C_b^{2p}\EE\left[(1+|Y_n^{k-1}|)^{2p}\right]+C(p)\frac{C_a^{2p}\EE[\zeta_1^{2p}]}{\sigma^{2p}}\EE\left[\left(1+|Y_n^{k-1}|\right)^{2p}\right]\right)\\
  & \leq \max\left(C_b^{2p},C(p)\frac{C_a^{2p}\EE[\zeta_1^{2p}]}{\sigma^{2p}}\right) 2^{2p} \frac{(j-i)^{p-1}}{n^{p}} \left\lbrace (j-i)+\sum_{k=i+1}^j\left(\EE\left[|Y_n^{k-1}|)^{2p}\right]\right)\right\rbrace\\
  & \leq \max\left(C_b^{2p},C(p)\frac{C_a^{2p}\EE[\zeta_1^{2p}]}{\sigma^{2p}}\right) 2^{2p} \exp\left(\sum_{k=i+1}^{j}2^{2p} \frac{(j-i)^{p-1}}{n^{p}}\right) (t_j-t_i)^p\\& \leq  \max\left(C_b^{2p},C(p)\frac{C_a^{2p}\EE[\zeta_1^{2p}]}{\sigma^{2p}}\right) 2^{2p}  \exp\left( 2^{2p}\right) (t_j-t_i)^p:=\Ck(p) (t_j-t_i)^p,
\end{align*}
where we have used the discrete version of the BDG inequality~\cite[Theorem 6.3]{BS15} 
in the martingale term $\sum_{k=i+1}^j\frac{1}{\sigma\sqrt{n}}a\left(Y_n^{k-1}\right)\zeta_{k}$ with 
$C(p):=6^p(p-1)^{p-1}$. 
Indeed, for the discrete-time martingale process $(x^{i,j}_n)_u:=\sum_{k=1}^u\frac{1}{\sigma\sqrt{n}}a\left(Y_n^{k+i-1}\right)\zeta_{i+k}$ for $u\in\{1,...,j-i\}$, we have 
$|(x^{i,j}_n)_{j-i}|\leq \displaystyle\max_{u\in \{1,.., j-i\}} |(x^{i,j}_n)_u|$ and the BDG inequality clearly also applies to $|x^{i,j}_{j-i}|$. 
We also used independence of~$\zeta_k$ and~$Y_{k-1}$ and the linear growth of~$b(\cdot)$ 
and~$a(\cdot)$ from Assumption~\ref{assu:CoeffSDE}, 
H\"older inequality and the discrete version  of Gronwall's lemma~\cite{Clark87} in the last step. 
Since $\EE[\zeta_k^{2p}]$ is bounded by Assumption~\ref{assu:xi} and the constant $\Ck(p)$ only depends on~$p$, but not on~$n$, then the tightness criterion~\eqref{eq:IneqZ} of Theorem ~\ref{thm:HolderConvergencePartition}  holds for $p>1$ with $\gamma=2p$ and $\beta=p-1$.
\end{proof}

\begin{corollary}\label{cor:Y_and_B}
Let $(Y_n)_{n\geq 1}$ be defined as in Theorem \ref{thm: Donsker diffusion} with innovations $(\zeta_i)_{i\geq 1}$, and suppose $(B^n)_{n\geq 1}$ is defined by the Donsker approximations \eqref{eq:Donsker}, for some innovations  $(\xi_i)_{i\geq 1}$ satisfying Assumption \ref{assu:xi} such that $((\zeta_i,\xi_i))_{i\geq 1}$ is iid~with $\mathrm{corr}(\zeta_i,\xi_i)=\rho$, for all $i\geq1$. Then there is joint weak convergence of $(B_n,Y_n)$ to $(B,Y)$ in $\left(\Cc^\lambda(\TT,\mathbb{R}^2),\|\cdot\|_\lambda\right)$, for all $\lambda<\frac{1}{2}$, for a standard Brownian motion $B$ such that $[B,W]_t=\rho t$, for $t\in \mathbb{I}$, where $W$ is the standard Brownian motion driving the dynamics of the weak limit $Y$ in \eqref{eq:diffusion}.
\end{corollary}
\begin{proof}
Take $(\zeta^\perp_i)_{i\geq 1}$ to satisfy Assumption \ref{assu:xi} and be independent of the innovations $(\zeta_i)_{i\geq 1}$ defining $(Y^n)_{n\geq1}$. Then set $\xi_i:=\rho \zeta_i + \sqrt{1-\rho^2} \zeta^\perp_i$, for $i\geq 1$, and let $B_n$ be defined in terms of $(\xi_i)_{i\geq 1}$. This yields the same finite-dimensional distributions of $(B_n,Y_n)$ as for the general $(\xi_i)_{i\geq 1}$ in the statement of the corollary. Consider now the drift vector $\textbf{b}(y)=(0,b(y))$ and the $2\times 2$ diffusion matrix $\textbf{a}(y)$ with rows $(\rho, \sqrt{1-\rho^2})$ and $(0,a(y))$. Then  Kushner \cite{Kushner74} applies directly to give finite-dimensional convergence with the desired limit. Finally, tightness of $(B_n,Y_n)$ follows analogously to the proof of Theorem \ref{thm: Donsker diffusion}. Hence the claim follows.
\end{proof}

\subsection{Invariance principle for rough processes}\label{sect:invariance_principles}
We have set the ground to extend our results to processes that are not necessarily $(1/2-\eps)$-H\"older continuous, Markovian nor semimartingales. 
More precisely, we are interested in $\alpha$-H\"older continuous paths with $\alpha\in(0,1)$, 
such as Riemann-Liouville fractional Brownian motion or some $\TBSS$ processes.
A key tool is the Continuous Mapping Theorem, first proved by Mann and Wald~\cite{MW43}, 
which establishes the preservation of weak convergence under continuous operators.

\begin{theorem}[Continuous Mapping Theorem]\label{thm: continuous mapping}
Let $(\Xx,\|\cdot\|_{\Xx})$ and $(\Yy,\|\cdot\|_{\Yy})$ be two normed spaces and assume that 
$g:\Xx\to \Yy$ is a continuous operator.
If the sequence of random variables $(Z_n)_{n\geq 1}$ converges weakly to~$Z$ in $(\Xx,\|\cdot\|_{\Xx})$,
then $(g(Z_n))_{n\geq 1}$ also converges weakly to~$g(Z)$ in $(\Yy,\|\cdot\|_{\Yy})$.
\end{theorem}
Many authors have exploited the combination of Theorems~\ref{thm: Donsker} 
and~\ref{thm: continuous mapping} to prove weak convergence~\cite[Chapter IV]{Pollard84}. 
This path avoids the lengthy computations of tightness and finite-dimensional convergence
in classical proofs~\cite{Billingsley68}. 
In fact, Hamadouche~\cite{Hamadouche00} already realised that Riemann-Liouville fractional operators 
are continuous, hence Theorem~\ref{thm: continuous mapping} 
holds under mapping by H\"older continuous functions. 
In contrast,  the novelty here is to consider the family of GFO applied to Brownian motion
together with the extension of Brownian motion to It\^o diffusions.
In fact, minimal changes to the proof of Proposition~\ref{prop: mapping BM} yield the following:
\begin{corollary}\label{cor:GalphaY}
If~$Y$ solves~\eqref{eq:diffusion}, then 
$\displaystyle (\Gg^\alpha Y)(t)=\int_0^t g(t-s)\D Y_s$
almost surely for all $t\in\TT$ and $\alpha\in\mathfrak{R}^{\frac{1}{2}}$.
\end{corollary}
The analogue of Theorem~\ref{thm: Donsker diffusion} for~$\Gg^\alpha Y$ holds as follows: 
\begin{theorem}[Generalised rough Donsker]\label{thm: GFO weak convergence}
For $(Y_n)$ in~\eqref{eq:Donskerdiffusion},
$Y$ its weak limit in $\left(\Cc^{\lambda}(\TT),\|\cdot\|_{\lambda}\right)$ for $\lambda<\frac{1}{2}$,
then the representation
\begin{equation}\label{eq:brute-force rDonsker}
\left(\Gg^\alpha Y_n\right)(t) = 
\displaystyle  \sum_{i=1}^{\nt}n\bigl[G(t-t_{i-1})-G(t-t_{i})\bigr]
 \left(Y_n^{i}-Y_n^{i-1}\right)
  + n \,G(t-t_{\nt})\bigl(Y_n(t)-Y_n^{\nt}\bigr),\quad  t\in\TT,
\end{equation}
holds.
Furthermore 
this sequence $\left(\Gg^\alpha Y_n\right)_{n\geq 1}$
converges weakly to~$\Gg^\alpha Y$ 
in $\left(\Cc^{\alpha+\lambda}(\TT),\|\cdot\|_{\alpha+\lambda}\right)$ for any $\alpha\in\Rkl$.
\end{theorem}
\begin{proof}
Recall that the sequence~\eqref{eq:Donskerdiffusion} is piecewise differentiable in time.
For $\alpha\in\Rkl_+$, note that $g(0)=0$ and therefore by integration by parts~\cite[Section 2.4]{Weinstock74} 
(where $Y_n$ is piecewise differentiable), for $n\geq 1$ and $t\in\TT$,
\begin{align*}
(\Gg^\alpha Y_n)(t)
 & = \int_0^t g'(t-s)(Y_n(s)-Y_n(0))\D s = \int_0^t g(t-s)\frac{\D (Y_n(s) - Y_n(0))}{\D s}\D s\\
 & = \frac{1}{\sigma\sqrt{n}}\left[\sum_{i=1}^{\nt}n\int_{t_{i-1}}^{t_{i}} g(t-s) a\left(Y_n^{i-1}\right)\zeta_{i}\D s
  + n\int_{t_{\nt}}^{t} g(t-s)a\left(Y_n^{\nt}\right) \zeta_{\nt+1}\D s\right]\\
 &+\frac{1}{n}\left[n\sum_{i=1}^{\nt}\int_{t_{i-1}}^{t_{i}} g(t-s) b\left(Y_n^{i-1}\right)\D s
  + n\int_{t_{\nt}}^{t} g(t-s)b\left(Y_n^{\nt}\right) \D s \right]  \\
  &=\sum_{i=1}^{\nt}n\left[G(t-t_{i-1}) - G(t-t_{i})\right] \left(Y_n^{i} - Y_n^{i-1}\right)
  + n (G(t-t_{\nt}) -  G(0))\bigl(Y_n(t)-Y_n^{\nt}\bigr),
\end{align*}
and~\eqref{eq:brute-force rDonsker} follows since $G(0)=0$ in the last line.
When $\alpha\in\Rkl_-$, using $G(0)=0$, we similarly get
\begin{align*}
 \int_{0}^{t} & g(t-s) (Y_n(s) - Y_n(0))\D s
 = \int_0^t G(t-s)\frac{\D (Y_n(s) - Y_n(0))}{\D s}\D s \\
   & 
   = \frac{1}{\sigma\sqrt{n}}
   \left[\sum_{i=1}^{\nt}n\int_{t_{i-1}}^{t_{i}} G(t-s) a\left(Y_n^{i-1}\right)\zeta_{i}\D s
  + n\int_{t_{\nt}}^{t} G(t-s)a\left(Y_n^{\nt}\right) \zeta_{\nt+1}\D s\right]\\
 &+\frac{1}{n}
 \left[n\sum_{i=1}^{\nt}\int_{t_{i-1}}^{t_{i}} G(t-s) b\left(Y_n^{i-1}\right)\D s
  + n\int_{t_{\nt}}^{t} G(t-s)b\left(Y_n^{\nt}\right) \D s \right] \\
  &=
  n\left\{
  \sum_{i=1}^{\nt}
  \left[\frac{b\left(Y_n^{i-1}\right)}{n} + \frac{a\left(Y_n^{i-1}\right)}{\sigma\sqrt{n}}\zeta_{i}\right]
  \int_{t_{i-1}}^{t_i}G(t-s)\D s
   + 
   \left[\frac{b\left(Y_n^{\nt}\right)}{n}
   + \frac{a\left(Y_n^{\nt}\right)}{\sigma\sqrt{n}}\zeta_{\nt+1}\right]
   \int_{t_{\nt}}^{t} G(t-s) \D s
   \right\}
  \\
  &=
  n\left\{
  \sum_{i=1}^{\nt}
  \left(Y_n^{i} - Y_n^{i-1}\right)
  \int_{t_{i-1}}^{t_i}G(t-s)\D s
   + 
   \left(Y_n(t)-Y_n^{\nt}\right)
   \int_{t_{\nt}}^{t} G(t-s) \D s
   \right\},
\end{align*}
and from there it follows readily that
\begin{align*}
(\Gg^\alpha Y_n)(t)&= \frac{\D}{\D t} \int_{0}^{t} g(t-s) (Y_n(s) - Y_n(0))\D s\\
   & =\sum_{i=1}^{\nt}n\left[G(t-t_{i-1})-G(t-t_{i})\right]  \left(Y_n^{i} - Y_n^{i-1}\right)
  + n G(t-t_{\nt})\bigl(Y_n(t)-Y_n^{\nt}\bigr),
\end{align*}
as desired (when $t=\frac{k}{n}$ the difference quotients pick up an extra term, but this vanishes in the limit).
Finally, the claimed convergence
follows analogously to that in Theorem~\ref{thm: Donsker diffusion} by continuous mapping, along with the fact that
$\Gg^\alpha$ is a continuous operator from $\left(\Cc^{\lambda}(\TT),\|\cdot\|_{\lambda}\right)$ 
to $\left(\Cc^{\lambda+\alpha}(\TT),\|\cdot\|_{\lambda+\alpha}\right)$ for all $\lambda\in(0,1)$ 
 and $\alpha\in\Rkl$.
\end{proof}
Notice here that the mean value theorem implies
\begin{equation}\label{eq:left-point rDonsker}
\left(\Gg^\alpha Y_n\right)(t)
 = \sum_{i=1}^{\nt}g\bigl(t_i^*\bigr) \bigl(Y_n^{i}-Y_n^{i-1}\bigr)
  + g\bigl(t^*_{\nt+1}\bigr)\bigl(Y_n(t)-Y_n^{\nt}\bigr),
\end{equation}
where $t^*_i\in[t-t_{i},t-t_{i-1}]$ and $t^*_{\nt+1}\in[0,t-t_{\nt}]$ and we use that $G(0)=0$. 
This expression is closer to the usual left-point forward Euler approximation. 
For numerical purposes,~\eqref{eq:left-point rDonsker} is much more efficient, since the integral~$G$ required in~\eqref{eq:brute-force rDonsker} is not necessarily available in closed form. 
Nevertheless, not any arbitrary choice of~$t^*_i$ gives the desired convergence from the above argument. We shall present a suitable candidate for optimal~$t_i^*$ in Section~\ref{sec:MomentMatching}, which guarantees weak convergence in H\"older sense. 

As could be expected, the Hurst parameter influences the speed of convergence of the scheme. 
We leave a formal proof to further study, but the following argument provides some intuition
about the correct normalising factor: 
Given $g \in \Ll^\alpha$, we can write $g(u)=u^{\alpha}L(u)$, where $L$ is a bounded function on $\mathbb{I}$. At time $t=t_i$, take $t^*_k=t_i-t_k+\frac{\varepsilon}{n}$ for $\varepsilon\in[0,1]$. For $\alpha\in\Rkl_-$, since $g \in \Ll^\alpha$, we can rewrite the approximation~\eqref{eq:left-point rDonsker} as
$$
\left(\Gg^\alpha Y_n\right)(t_i)
 = \frac{1}{n^{1/2+\alpha}}\sum_{k=1}^{i}\left(i - k + \varepsilon \right)^\alpha L(t_k^*)\left(Y^k_n-Y_n^{k-1}
 \right)\sqrt{n},\quad \text{for }i=0,\ldots,n.
$$
Here, $(i-k+\varepsilon)^\alpha\leq \varepsilon^\alpha$ is bounded in $n\geq1$ as long as $\varepsilon\in (0,1]$, so the normalisation factor is of order $n^{-\alpha-1/2}$.
When $\alpha\in\Rkl_+$, the approximation~\eqref{eq:left-point rDonsker} instead reads as
$$
\left(\Gg^\alpha Y_n\right)(t_i)
 = \frac{1}{\sqrt{n}}\sum_{k=1}^{i}\bigl(t_i-t_k+\tfrac{\varepsilon}{n}\bigr)^\alpha L(t^*_k)\left(Y^k_n-Y_n^{k-1}\right)\sqrt{n},\quad \text{for }i=0,\ldots,n,
$$
in which case $(t_i-t_k+\tfrac{\varepsilon}{n})^\alpha\leq t_i^\alpha$ is bounded in $n\geq 1$, 
and hence the normalisation factor is of order~$n^{-1/2}$.
This intuition is consistent with the result by Neuenkirch and Shalaiko~\cite{NS16}, 
who found the strong rate of convergence of the Euler scheme to be of order $\mathcal{O}(n^{-H})$ for $H<\frac{1}{2}$ for fractional Ornstein-Uhlenbeck. 
So far, our results hold for $\alpha$-H\"older continuous functions;
however, for practical purposes, it is often necessary
to constrain the volatility process~$(V_t)_{t\in\TT}$ to remain strictly positive at all times. 
The stochastic integral~$\Gg^\alpha Y$ need not be so in general.
However, a simple transformation (e.g. exponential) can easily overcome this fact. 
The remaining question is whether the $\alpha$-H\"older continuity 
is preserved after this composition. 
\begin{proposition}\label{prop: weak conv family Holder continuous}
Let $(Y_n)_{n\geq 1}$ be the approximating sequence~\eqref{eq:Donskerdiffusion} in~$\Cc^{\lambda}(\TT)$
for $\lambda<1/2$.
Then $\left(\Phi\left(\Gg^\alpha Y_n\right)\right)$ converges weakly to~$\Phi\left(\Gg^\alpha Y\right)$ 
in $\left(\Cc^{\alpha+\lambda}(\TT),\|\cdot\|_{\alpha+\lambda}\right)$ for all $\alpha\in\Rkl$.
\end{proposition}

\begin{proof}
Theorem~\ref{thm: GFO weak convergence} gives that $ \Gg^\alpha Y_n$ converges weakly to $\Gg^\alpha Y$ in $\left(\Cc^{\lambda+\alpha}(\TT),\|\cdot\|_{\lambda+\alpha}\right)$. By our assumptions, $\Phi$ is continuous from $\left(\Cc^{\lambda+\alpha}(\TT),\|\cdot\|_{\lambda+\alpha}\right)$ to $\left(\Cc^{\lambda+\alpha}(\TT),\|\cdot\|_{\lambda+\alpha}\right)$. The proposition thus follows from the continuous mapping theorem. The diagram below summarises the steps with $\lambda<1/2$.
The double arrows show weak convergence, and we indicate next to them the topology in which it takes place.
\begin{center}
\begin{tikzpicture}
  \matrix (m) [matrix of math nodes, row sep=2em, column sep=3em, text height=1ex, text depth=0.25ex,nodes={anchor=center}]
  {
     \left(\Cc^{\lambda}(\TT),\|\cdot\|_{\lambda}\right) & \left(\Cc^{\alpha+\lambda}(\TT),\|\cdot\|_{\alpha+\lambda}\right) & \left(\Cc^{\alpha+\lambda}(\TT),\|\cdot\|_{\alpha+\lambda}\right)  \\
     \displaystyle Y_n & \displaystyle\Gg^{\alpha}(Y_n) & \Phi(\Gg^\alpha Y_n) \\
     \text{} & \text{} & \text{} \\ 
     Y & \Gg^\alpha Y&  \Phi(\Gg^\alpha Y)\\};
  \path[->,very thick,shorten <=10pt,shorten >=10pt]
    (m-1-1.north) edge[bend left =20] node [above] {$\Gg^{\alpha}$}  (m-1-2.north)
     (m-2-1) edge[bend left =0] node [above] {$\Gg^{\alpha}$}  (m-2-2)
     (m-4-1) edge[bend left =0] node [above] {$\Gg^{\alpha}$}  (m-4-2)
     (m-1-2.north) edge[bend left =20] node [above] {$\Phi$}  (m-1-3.north)
     (m-2-2) edge[bend left =0] node [above] {$\Phi$}  (m-2-3)
     (m-4-2) edge[bend left =0] node [above] {$\Phi$}  (m-4-3);
     
\path[->,very thick,shorten <=10pt,shorten >=10pt]
    (m-2-1) edge[double] node [left] {$\|\cdot\|_{\lambda}$}  (m-4-1)
    (m-2-2) edge[double] node [left] {$\|\cdot\|_{\alpha+\lambda}$}  (m-4-2)
     (m-2-3) edge[double] node [left] {$\|\cdot\|_{\alpha+\lambda}$}  (m-4-3);
\end{tikzpicture}
\end{center}
\end{proof}


\subsection{Extending the weak convergence to the Skorokhod space and proof of Theorem~\ref{thm:MainThm}}\label{sec:ProofMainThm}
The Skorokhod space of c\`adl\`ag processes equipped with the Skorokhod topology 
has been widely used to prove weak convergence~\cite{Billingsley68,Ethier_Kurtz}. 
The Skorokhod space of c\`adl\`ag processes equipped with the Skorokhod norm,
which we denote $\left(\Dd(\TT),d_{\Dd}\right)$, 
markedly simplifies when we only consider continuous processes (as is the case of our framework with H\"older continuous processes). 
Billingsley~\cite[Chapter~3, Section 12]{Billingsley68} proved that the identity 
$\left(\Dd(\TT)\cap\Cc(\TT),d_{\Dd}\right)=\left(\Cc(\TT),\|\cdot\|_{\infty}\right)$
always holds.
This seemingly simple statement allows us to reduce proofs of weak convergence of continuous processes 
in the Skorokhod topology to that in the supremum norm, usually much simpler.
We start with the following straightforward observation:
\begin{lemma}
For any $\lambda\in(0,1)$, the identity map is continuous from $\left(\Cc^{\lambda}(\TT),\|\cdot\|_\lambda\right)$ 
to $\left(\Dd(\TT),d_{\Dd}\right)$.
\end{lemma}
\begin{proof}
Since the identity map is linear, it suffices to check that it is bounded. 
For this observe that 
$\| f\|_{\lambda}=|f|_{\lambda}+\sup_{t\in\TT} |f(t)|=|f|_{\lambda}+\| f\|_{\infty}>\| f\|_{\infty}$,
where $|f|_{\lambda}>0$, which concludes the proof since the Skorokhod norm in the space of continuous functions is  equivalent to the supremum norm.
\end{proof}
Applying the Continuous Mapping Theorem twice, 
first with the Generalised fractional operator (Theorem~\ref{thm: GFO weak convergence}), 
then with the identity map, yields the following result directly:
\begin{theorem}\label{thm: weak conv GFO Skorokhod}
For any $\alpha \in \Rk^{1/2}$, the sequence $\left(\Phi(\Gg^\alpha Y_{n})\right)$ converges weakly to $\Phi\left(\Gg^\alpha Y\right)$ in $\left(\Dd(\TT),d_{\Dd}\right)$. Moreover, the sequence is tight in $\left(\Cc(\TT),\|\cdot\|_{\infty}\right)$.
\end{theorem}
The final step in the proof of our main theorem
is to extend the functional weak convergence to the log-stock price $X$. 
For this, we will rely on the weak convergence theory for stochastic integrals due to Jakubowski, Memin and Pag\`es~\cite{JMP89} and further developed by Kurtz and Protter~\cite{KP91}. Throughout, we write
$H \bullet N:=\int_0^\cdot H(s)\D N(s)$ and we use the notation $H^-$ for the process $H^-(t):=H(t-)$ obtained by taking left limits.
\begin{theorem}[Kurtz and Protter~\cite{KP91}]\label{thm:KurtzProtter}
For each $n\geq 1$, let $N_n=M_n+A_n$ be an $(\mathcal{F}^n_t)$-semimartingale and let
$H_n$ be an $(\mathcal{F}^n_t)$-adapted c\`adl\`ag process on~$\TT$. Suppose that, for all $\gamma>0$, there are $(\mathcal{F}^n_t)$-stopping times $(\tau_n^\gamma)$ such that $\sup_{n\geq 1}\mathbb{P}(\tau_n^\gamma \leq \gamma ) \leq 1/\gamma$ and $\sup_{n\geq1}\mathbb{E}[ [M_n]_{\tau_n^\gamma\land 1}+T_{ \tau_n^\gamma\land 1}(A_n)]<\infty$, where $T_t$ denotes the total variation on $[0,t]$. If $\left(H_n,N_n\right)$ converges weakly
to~$(H,N)$ in $(\Dd(\TT, \mathbb{R}^2),d_{\Dd})$, then $N$ is a semimartingale in the filtration generated by $(H,N)$ and 
$\left(H_n, N_n, H_{n}^-\bullet N_n\right)$ converges weakly to
$(H, N, H^- \bullet N)$ in $(\Dd(\TT, \mathbb{R}^3),d_{\Dd})$.
\end{theorem}

The above amounts to a restatement of~\cite[Theorem 2.2]{KP91} in the special case $\delta=\infty$ 
(in their notations) and restricted to real-valued processes on $\TT$. With this, we can now give the proof of Theorem~\ref{thm:MainThm}, which asserts the functional weak convergence of the approximations $X_n$ from~\eqref{eq:Donsker log stock} to the desired log-price~$X$ from~\eqref{eq:System}.

\begin{proof}[Proof of Theorem~\ref{thm:MainThm}]
We begin by considering, for all $n\geq 1$, the particular approximations
\[
M_n(t):= \frac{1}{\sigma\sqrt{n}}\sum_{i=1}^{\nt}\xi_i,
\]
$t\in \TT$, of the driving Brownian motion $B$ in the dynamics of $X$. Here the $\xi_i$ satisfy Assumption~\ref{assu:xi} and so does the $\zeta_i$ in the construction of~$Y_n$ from~\eqref{eq:Donskerdiffusion}. 
While each pair~$\xi_i$ and~$\zeta_i$ are correlated, they form an iid sequence $\{(\zeta_i,\xi_i)\}_{i\geq 1}$ across the pairs. In particular, it is straightforward to see that each $M_n$ is a martingale on $\TT$ for the filtration $(\mathcal{F}^n_t)$ defined by $\mathcal{F}^n_t:=\sigma(\zeta_i,\xi_i:i=1,\ldots,\nt)$. Moreover, we have
\[
\mathbb{E}\bigl[[M_n]_t\big] = \nt \frac{1}{\sigma^2n }\mathbb{E}[\xi_1^2] =\frac{\nt}{n} \leq 1,
\]
for $t\in \TT$, for all $n\geq1$. Consequently, we can simply take $\tau_n^\gamma:=+\infty$, for all $\gamma >0$ and $n\geq1$, to satisfy the required control on the integrators $N_n:=M_n$ in Theorem~\ref{thm:KurtzProtter}. By~\cite[Chapter~7, Theorem 1.4]{Ethier_Kurtz}, the $M_n$ converge weakly to a Brownian motion $B$ in $(\Dd(\TT),d_{\Dd})$. Now fix  $\alpha \in \Rk^{\frac{1}{2}}$ and define a sequence of c\`adl\`ag processes $H_n$ on $\TT$, for all $n\geq 1$, by setting $H_n(1):=\Phi(\Gg^\alpha Y_n)(1)$ and $H_n(t):=\Phi(\Gg^\alpha Y_{n})(t_{k-1})$ for $t\in[t_{k-1},t_k)$, for each $k=1,\ldots,n$. In view of Theorem~\ref{thm: weak conv GFO Skorokhod}, the Arzela--Ascoli characterisation of tightness~\cite[Theorem 8.2]{Billingsley68} for the space $\left(\Cc(\TT),\|\cdot\|_{\infty}\right)$ allows us to conclude that the $H_n$ converge weakly to $H:=\Phi(\Gg^\alpha Y)$ in $(\Dd(\TT),d_{\Dd})$. Furthermore, recalling the definition of $Y_n$ in~\eqref{eq:Donskerdiffusion}, each~$H_n$ is adapted to the filtration $(\mathcal{F}^n_t)$ introduced above. By Corollary~\ref{cor:Y_and_B}, we readily deduce that there is joint weak convergence of $(Y_n,H_n,M_n)$ to $(Y,H,B)$ on $(\Dd(\TT),d_\mathcal{D})\times (\Dd(\TT),d_\mathcal{D})\times (\Dd(\TT),d_\mathcal{D}) $, where $Y$ satisfies~\eqref{eq:diffusion} for a Brownian motion $W$ with $[ W, B ]_t = \rho t$, for all $t\in \mathbb{I}$. As noted in~\cite{KP91}, the Skorokhod topology on $\Dd(\TT,\mathbb{R}^2)$ is stronger than the product topology on $\Dd(\TT)\times\Dd(\TT)$, but here it automatically follows that we have weak convergence of the pairs $(H_n,M_n)$ to $(H,B)$ in $\left(\Dd(\TT, \mathbb{R}^2),d_{\Dd}\right)$, by standard properties of the Skorokhod topology (e.g.~\cite[Chapter~3, Theorem 10.2]{Ethier_Kurtz}), since the limiting pair $(H,B)$ is continuous. Consequently, we are in a position to apply Theorem \ref{thm:KurtzProtter}.
To this end, observe that
$$(H^-_n \bullet M_n)(t) = \sum_{k=1}^{\nt} \!H_n(t_{k}-)(M_n(t_{k})-M_n(t_{k}-)) =\frac{1}{\sigma \sqrt{n}}\sum_{k=1}^{\nt} \!\Phi(\Gg^\alpha Y_{n})(t_{k-1})\xi_k,
$$
which is precisely the second term on the right-hand side of~\eqref{eq:Donsker log stock}. Therefore, Theorem~\ref{thm:KurtzProtter} gives that the stochastic integral $H \bullet M = \sqrt{\Phi(\Gg^\alpha Y)}\bullet B$ is the weak limit in $(\Dd(\TT),d_{\Dd})$ of the second term on the right-hand side of~\eqref{eq:Donsker log stock}. For the first term on the right-hand side of~\eqref{eq:Donsker log stock}, we have $-\frac{1}{2}\int_0^\cdot H_n(s) \D s$ 
converging weakly to 
$-\frac{1}{2}\int_0^\cdot H(s)\D s$, 
by the continuous mapping theorem, as the integral is a continuous operator 
from $\left(\Dd(\TT),d_\Dd\right)$ to itself. Since there is weak convergence of the pairs $(H_n,H^-_n \bullet M_n)$ to $(H, H \bullet B)$ in $(\Dd(\TT,\mathbb{R}^2),d_{\Dd})$, the sum of the two terms on the right-hand side of~\eqref{eq:Donsker log stock} are then also weakly convergent in $\left(\Dd(\TT),d_\Dd\right)$. Recalling that the limit $Y$ satisfies~\eqref{eq:diffusion} for a Brownian motion $W$ such that $W$ and $B$ are correlated with parameter $\rho$, we hence conclude that $X$ converges in $\left(\Dd(\TT),d_\Dd\right)$ to the desired limit.
\end{proof}

\section{Applications}\label{sec:Applications}

\subsection{Weak convergence of the Hybrid scheme}
The Hybrid scheme (and its turbocharged version~\cite{MP18}) 
introduced by Bennedsen, Lunde and Pakkanen~\cite{BLP17} 
is the current state-of-the-art to simulate~$\TBSS$ processes. 
However, only convergence in mean-square-error was proved, but not (functional) weak convergence,
which would justify the use of the scheme for path-dependent options.
Unless otherwise stated, we shall denote by  $\Tt_n:=\{t_k = \frac{k}{n}\}_{k=0,...,n}$
the uniform grid on~$\TT$.
We show that the H\"older convergence also holds for the case $g(x)=x^\alpha$:
\begin{proposition}\label{prop: proof Hybrid scheme}
The Hybrid scheme sequence $(\widetilde{\Gg}^\alpha W_n)$ defined as 
\begin{equation}\label{eq:Hybrid}
\widetilde{\Gg}^\alpha W_n(t):=\widetilde{\Gg}^\alpha W_n\left(\frac{\nt}{n}\right) +(nt-\nt)\left(\widetilde{\Gg}^\alpha W_n\left(\frac{\nt+1}{n}\right) -\widetilde{\Gg}^\alpha W_n\left(\frac{\nt}{n}\right) \right),
\end{equation}
for $t\in\TT$, where
\begin{equation}\label{eq:Hybrid aux}
\widetilde{\Gg}^\alpha W_n(t_i)
 := \sum_{k=1}^{(i-\kappa)\vee 0}\int_{t_{k-1}}^{t_k}\sqrt{n} (t_i-s)^\alpha\D s\xi_{k} + \int_{0\vee t_{i-\kappa}}^{t_{i}}(t_{i}-s)^\alpha\D W_s,\quad i = 0,\ldots,n,\quad\kappa\geq 1.
\end{equation}
with $\xi_k := \int_{t_{k-1}}^{t_k}\D W_s\sim \mathcal{N}(0,1/n)$, converges to~$\Gg^\alpha W$ 
in $\left(\Cc^{\alpha+1/2},\|\cdot\|_{\alpha+1/2}\right)$ for $\alpha \in \Rk^{1/2}$ and $\kappa=1$.
\end{proposition}
\begin{proof}
Finite-dimensional convergence follows trivially as the target process is centered Gaussian, thus convergence of the covariance matrix ensures finite-dimensional convergence. 
To prove convergence we only need to to show that the approximating sequence is tight, by verifying the criteria from Theorem~\ref{thm:HolderConvergencePartition} as follows:
$$
\EE\left[\left|\widetilde{\Gg}^\alpha W_n(t_i)-\widetilde{\Gg}^\alpha W_n(t_j)\right|^{2p}\right]\leq C |t_i-t_j|^{2p\alpha+p},
$$ 
for all $t_i, t_j \in \Tt_n$, 
for $p\geq 1$ and some constant $C\geq 0$.
Without loss of generality assume $t_j< t_i$ and take $\kappa =1$. Define
$$
\widetilde{\sigma}^2:=\EE\left[\left|\widetilde{\Gg}^\alpha W_n(t_i)-\widetilde{\Gg}^\alpha W_n(t_j)\right|^{2}\right].
$$
We note that
$$
\EE\left[n\left(\int_{t_{j-1}}^{t_{j}}(t_{i}-s)^\alpha\D s\int_{t_{j-1}}^{t_{j}}\D W_s-\int_{t_{j-1}}^{t_{j}}(t_{j}-s)^\alpha\D W_s\right)^2\right]\leq \int_{t_{j-1}}^{t_{j}}\left((t_{i}-s)^\alpha-(t_j-s)^\alpha\right)^2\D s,
$$
where we have used Chebyshev's integral inequality. Therefore,
\begin{align*}
\widetilde{\sigma}^2&\leq\sum_{k=1}^{j}\left(\int_{t_{k-1}}^{t_k} (t_i-s)^\alpha-(t_j-s)^\alpha\D s \right)^2+\int_{t_{j-1}}^{t_{j}}\left((t_{i}-s)^\alpha-(t_j-s)^\alpha\right)^2\D s + \int_{t_{j}}^{t_{i}}(t_{i}-s)^{2\alpha}\D s \\
&\leq \sum_{k=1}^{j-1}\left(\int_{t_{k-1}}^{t_k} (t_i-s)^{\alpha}-(t_j-s)^\alpha \D s \right)^2+ \int_{t_{j-1}}^{t_{j}}(t_i-t_j)^{2\alpha}\D s+\frac{1}{2\alpha+1}(t_i-t_j)^{2\alpha+1} \\
&\leq \sum_{k=1}^{j-1}\left(\int_{t_{k-1}}^{t_k} \left( t_i-t_j \right)^{\alpha}\D s \right)^2+ \frac{1}{n}(t_i-t_j)^{2\alpha}\D s \frac{1}{2\alpha+1}(t_i-t_j)^{2\alpha+1}\\
&\leq \frac{1}{n} \left( t_i-t_j \right)^{2\alpha}+ \frac{1}{n}(t_i-t_j)^{2\alpha}\D s\frac{1}{2\alpha+1}(t_i-t_j)^{2\alpha+1}\leq C(t_i-t_j)^{2\alpha+1} \end{align*}
where we have used the power inequality $|x|^p-|y|^p\leq |x-y|^p$ for $p\leq 1$.
Thus, by standard moment properties of Gaussian random variables~\cite[Theorem 2.1]{Boucheron13} we obtain 
$$\EE\left[\left|\widetilde{\Gg}^\alpha W_n(t_i)-\widetilde{\Gg}^\alpha W_n(t_j)\right|^{2p}\right]\leq \widetilde{C} \widetilde{\sigma}^{2p} \leq \widetilde{C}(t_i-t_j)^{2p\alpha+p},$$
which gives the desired result.
\end{proof}
 We further note that Proposition~\ref{prop: proof Hybrid scheme} and Theorem~\ref{thm:MainThm} ensure the weak convergence of the log-stock price for the Hybrid scheme as well.
\begin{remark}
Proposition~\ref{prop: proof Hybrid scheme} may easily be extended to a $d$-dimensional Brownian motion~$W$ (for example for multifactor volatility models), 
also providing a weak convergence result for the $d$-dimensional version of the Hybrid scheme 
recently developed by Heinrich, Pakkanen and Veraart~\cite{HPV19}.
\end{remark}

\subsection{Application to fractional binomial trees}
We consider a binomial setting for the Riemann-Liouville fractional Brownian motion 
$\Gg^{H-1/2} W$ with $g(u)\equiv u^{H-1/2}$, $H\in(0,1)$,
for which Theorem~\ref{thm: GFO weak convergence} provides a weakly converging sequence. 
On the partition~$\Tt_n$, with a Bernoulli sequence~$\{\zeta_i\}_{i=1}^n$ satisfying
$\PP(\zeta_i=1)=\PP(\zeta_i=-1)=\frac{1}{2}$ for all~$i$
(justified by Theorem~\ref{thm:MainThm}), 
the approximating sequence reads
$$
(\Gg^{H-1/2} W_n)(t_i)
 := \frac{1}{\sqrt{n}}\sum_{k=1}^{i}\left(t_{i} - t_{k-1} \right)^{H-1/2}\zeta_k,\quad\text{for } i=0,\ldots,n.
$$
Figure~\ref{bino_tree_H_75} shows a fractional binomial tree structure for $H=0.75$ and $H=0.1$. 
Despite being symmetric, such trees cannot be recombining due to the (non-Markovian) 
path-dependent nature of the process. 
It might be possible, in principle, to modify the tree at each step to make it recombining, 
following the procedure developed in~\cite{ADS14} for Markovian stochastic volatility models.
It is not so straightforward though, and requires a dedicated thorough analysis which we leave for future research.

\begin{figure}[h!]
\centering
\includegraphics[scale=0.5]{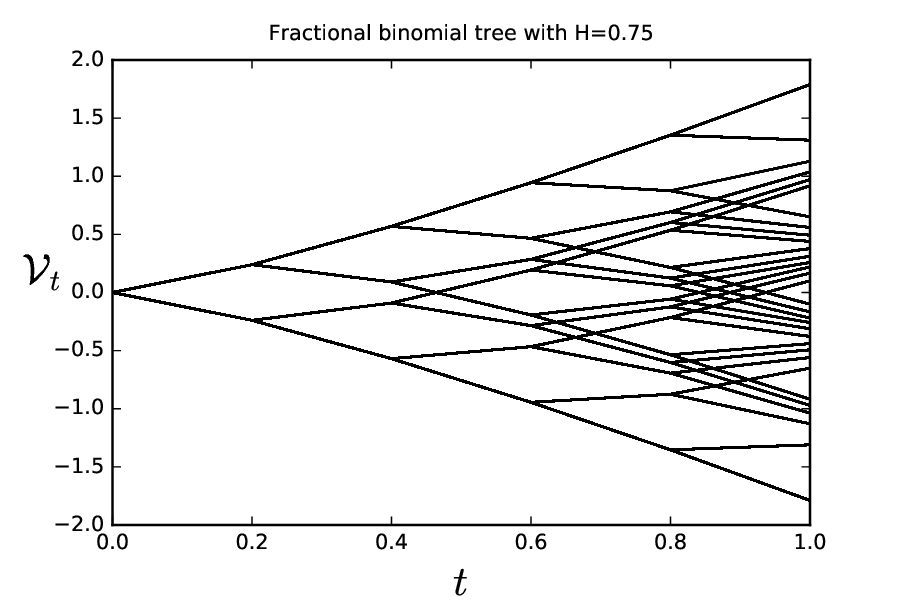}
\includegraphics[scale=0.5]{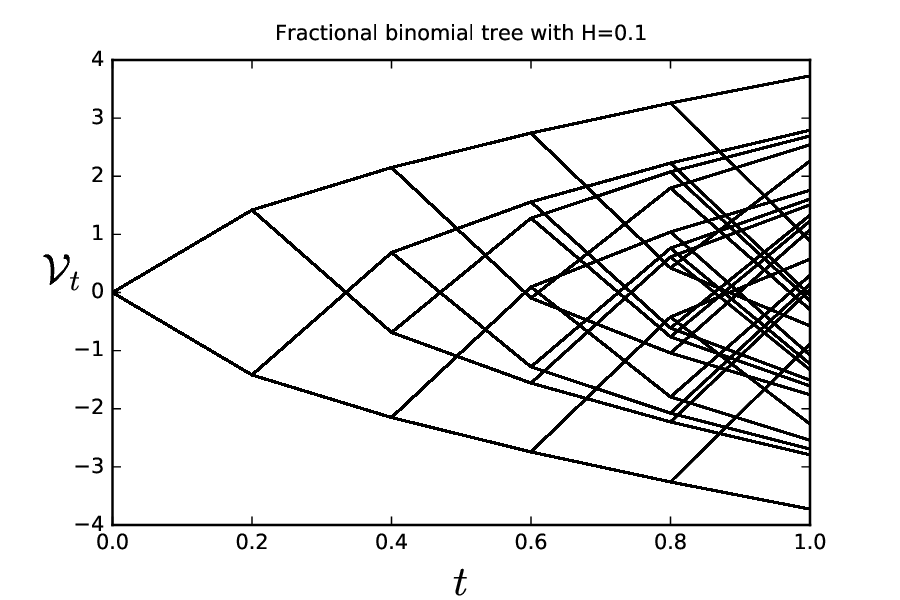}
\caption{Binomial tree for the Riemann-Liouville fractional Brownian motion with $n=5$ discretisation points
for $H=0.75$ (left) and $H=0.1$ (right).}
\label{bino_tree_H_75}
\end{figure}

\subsection{Monte-Carlo}
Theorem~\ref{thm:MainThm} introduces the theoretical foundations 
of Monte-Carlo methods (in particular for path-dependent options) for rough volatility models. 
In this section we give a general and easy-to-understand recipe to implement the class of rough volatility models~\eqref{eq:System}. For the numerical recipe to be as general as possible, we shall consider the general time partition $\Tt:=\{t_i=\frac{iT}{n}\}_{i=0,\ldots,n}$ on $[0,T]$ with $T>0$.
\begin{algorithm}[Simulation of rough volatility models]\label{algo:rDonsker}\ 
\begin{enumerate}
\item Simulate two $\Nn(0,1)$ matrices
$\{\xi_{j,i}\}_{\substack{j=1,\ldots,M\\i=1,\ldots,n}}$ and 
$\{\zeta_{j,i}\}_{\substack{j=1,\ldots,M\\i=1,\ldots,n}}$  
with $\text{corr}(\xi_{j,i},\zeta_{j,i})=\rho$;
\item \label{step 1} simulate M paths of~$Y_n$ via\footnote{Here, $Y_n^{j}(t_i)$ denotes the $j$-th path~$Y_n$ evaluated at the time point~$t_i$,
which is different from the notation~$Y_n^{j}$ in the theoretical framework above, 
but should not create any confusion.}
$$
Y_n^{j}(t_i)
 \displaystyle = \frac{T}{n}\sum_{k=1}^{i}b(Y_n^{j}(t_{k-1}))
  + \frac{T}{\sqrt{n}}\sum_{k=1}^{i}\sit\left(Y_n^{j}(t_{k-1})\right)\zeta_{j,k},
  \quad i=1,\ldots,n\text{ and }j=1,\ldots,M,
$$
 and also compute
$$
\Delta Y^{j}_n(t_i) := Y^{j}_n(t_{i}) - Y^{j}_n(t_{i-1}),
\quad i=1,\ldots,n \text{ and }j=1,\ldots,M,
$$
\item \label{step 2 frac} Simulate $M$ paths of the fractional driving process~$((\Gg^\alpha Y_n)(t))_{t \in \Tt}$ using 
$$
(\Gg^\alpha Y_n)^{j}\left(t_i\right)
 := \sum_{k=1}^{i}g(t_{i-k+1})\Delta Y^{j}_n(t_k)
 = \sum_{k=1}^{i}g(t_k)\Delta Y^{j}_n(t_{i-k+1}),
 \quad i=1,\ldots,n\text{ and }j=1,\ldots,M.
 $$
The complexity of this step is in general of order $\Oo(n^2)$ (see Appendix~\ref{sec: convolution} for details).
However, this step is easily implemented using discrete convolution with complexity $\Oo(n\log n)$ 
(see Algorithm~\ref{algo: convolution for rough vol} in Appendix~\ref{sec: convolution} for details in the implementation). 
With the vectors
$\gr := (g(t_i))_{i=1,\ldots,n}$
and
$\Delta Y^{j}_n := (\Delta Y_n^{j}(t_i))_{i=1,\ldots,n}$
for $j=1,\ldots,M$,
we can write
$(\Gg^\alpha Y_n)^j(\Tt)=\sqrt{\frac{T}{n}}(\gr\ast\Delta Y^j_n)$, for $j=1,\ldots,M$,
where~$\ast$ represents the discrete convolution operator.
\item  Use the forward Euler scheme to simulate the log-stock process,
for all $i=1,\ldots,n$, $j=1,\ldots,M$, as
$$
X^j(t_i) = X^j(t_{i-1})
 - \frac{1}{2}\frac{T}{n}\Phi\left(\Gg^\alpha Y_n\right)^j(t_{i-1})
 + \sqrt{\frac{T}{n}}\sqrt{\Phi\left(\Gg^\alpha Y_n\right)^j(t_{i-1})}\xi_{j,i}.
$$
\end{enumerate}
\end{algorithm}
\begin{remark}\ 
\begin{itemize}
\item
When $Y=W$, we may skip step~\eqref{step 1} and replace $\Delta Y^j_n(t_i)$ by $\sqrt{T/n}\zeta_{i,j}$ 
on step~\eqref{step 2 frac}.
\item 
Step~\eqref{step 2 frac} may be replaced by the Hybrid scheme algorithm~\cite{BLP17} only when $Y=W$.
\end{itemize}
\end{remark}

Antithetic variates in Algorithm~\ref{algo:rDonsker}
are easy to implement as it suffices to consider the uncorrelated random vectors
$\zeta_j := (\zeta_{j,1},\zeta_{j,2},\ldots,\zeta_{j,n})$
and 
$\xi_j := (\xi_{j,1},\xi_{j,2},\ldots,\xi_{j,n})$, for $j=1,\ldots,M$.
Then
$(\rho \xi_j + \rrho\zeta_j,\xi_j)$, 
$(\rho \xi_j - \rrho\zeta_j,\xi_j)$, 
$(-\rho \xi_j - \rrho\zeta_j,-\xi_j)$ and 
$(-\rho \xi_j + \rrho\zeta_j,-\xi_j)$,
for $j=1,\ldots,M$,
constitute the antithetic variates, which significantly improves the performance of the Algorithm~\ref{algo:rDonsker} by reducing memory requirements, reducing variance and accelerating execution by exploiting symmetry of the antithetic random variables.

\subsubsection{Enhancing performance}\label{sec:MomentMatching}
A standard practice in Monte-Carlo simulation is to match moments of the approximating sequence with the target process. 
In particular, when the process is Gaussian, matching first and second moments suffices. 
We only illustrate this approximation for Brownian motion:
the left-point approximation~\eqref{eq:left-point rDonsker} (with $Y=W$) 
may be modified to match moments as
\begin{equation}\label{eq:optimal eval}
(\Gg^\alpha W)(t_i) \approx \frac{1}{\sigma\sqrt{n}}\sum_{k=1}^{i}g\left(t^*_k\right)\zeta_{k},
\qquad \text{for }i=0,\ldots,n,
\end{equation}
where $t^*_k$ is chosen optimally.
Since the kernel~$g(\cdot)$ is deterministic, 
there is no confusion with the Stratonovich stochastic integral, 
and the resulting approximation will always converge to the It\^o integral. 
The first two moments of~$\Gg^\alpha W$ read
$$
\EE\left[\left(\Gg^\alpha W\right)(t)\right] = 0
\qquad\text{and}\qquad
\VV\left[\left(\Gg^\alpha W\right)(t)\right] = \int_0^tg(t-s)^2\D s.
$$
The first moment of the approximating sequence~\eqref{eq:optimal eval} is always zero,
and the second moment reads
$$
\VV\left(\frac{1}{\sigma\sqrt{n}}\sum_{k=1}^{j-1}g\left(t^*_k\right)\zeta_{k}\right)
 =\frac{1}{n}\sum_{k=1}^{j-1}g\left(t^*_k\right)^2.
$$
Equating the theoretical and approximating quantities we obtain
$\frac{1}{n}g(t^*_k)^2\D s=\int_{t_{k-1}}^{t_{k}} g(t-s)^2\D s$
for $k=1,\ldots,n$, 
so that the optimal evaluation point can be computed as
\begin{equation}\label{eq: optimal eval}
g(t_k^*)=\sqrt{n\int_{t_{k-1}}^{t_{k}} g(t-s)^2\D s},
\qquad \text{for any }k=1,\ldots,n.
\end{equation}
With the optimal evaluation point the scheme is still a convolution 
so that Algorithm~\ref{algo: convolution for rough vol} in Appendix~\ref{sec: convolution} 
can still be used for faster computations.
In the Riemann-Liouville fractional Brownian motion case, 
$g(u) = u^{H-1/2}$, and the optimal point can be computed in closed form as
$$
t_k^* = \left(\frac{n}{2H}\left[\left(t-t_{k-1}\right)^{2H}-\left(t-t_{k}\right)^{2H}\right]\right)^{1/(2H-1)},
\qquad \text{for each }k=1,\ldots,n.
$$
\begin{proposition}\label{prop: proof moment-match scheme}
The moment matching sequence $(\widehat{\Gg}^\alpha W_n)$ defined as 
\begin{equation}\label{eq:moment-match}
\widehat{\Gg}^\alpha W_n(t):=\widehat{\Gg}^\alpha W_n\left(\frac{\nt}{n}\right) +(nt-\nt)\left(\widehat{\Gg}^\alpha W_n\left(\frac{\nt+1}{n}\right) -\widehat{\Gg}^\alpha W_n\left(\frac{\nt}{n}\right) \right), t\in\TT,
\end{equation}
where
\begin{equation}\label{eq:moment-aux}
\widehat{\Gg}^\alpha W_n(t_i)
 := \sum_{k=1}^{i}\sqrt{n\int_{t_{k-1}}^{t_k} g(t_i-s)^2\D s}\xi_{k}.
\end{equation}
with $(\xi_k)$ an iid family of centered sub-Gaussian random variables with 
$\EE[\xi_k^2]=\frac{1}{n}$
(namely $\PP(|\xi_k|>x)\leq C \E^{-vx^2}$ for all $x>0$ and some $C,v>0$). 
Then, convergence to~$\Gg^\alpha W$ 
holds in $\left(\Cc^{\alpha+1/2},\|\cdot\|_{\alpha+1/2}\right)$ for $\alpha \in \Rk^{1/2}$.
\end{proposition}
\begin{proof}
Finite-dimensional convergence follows from the Central Limit Theorem as the target process is centered Gaussian, 
thus convergence of the covariance matrix ensures finite-dimensional convergence. It then suffices to prove that the approximating sequence is tight in the desired space, which, in view of Theorem~\ref{thm:HolderConvergencePartition}, can be deduced by establishing the control
$$
\EE\left[\left|\widehat{\Gg}^\alpha W_n(t_i)-\widetilde{\Gg}^\alpha W_n(t_j)\right|^{2p}\right]\leq C |t_i-t_j|^{2p\alpha+p},
$$
for all $t_i,t_j\in\mathcal{T}_n$, for $p\geq 1$ and some constant $C\geq 0$. 
We have
$$
\widetilde{\sigma}^2:=\EE\left[\left(\widehat{\Gg}^\alpha W_n(t_i)-\widehat{\Gg}^\alpha W_n(t_j)\right)^2\right]=\EE\left[\left(\widehat{\Gg}^\alpha W_n(t_i)\right)^2\right]+\EE\left[\left(\widehat{\Gg}^\alpha W_n(t_j)\right)^2\right]-2\EE\left[\widehat{\Gg}^\alpha W_n(t_i)\widehat{\Gg}^\alpha W_n(t_j)\right].
$$
We note that
$$
\EE\left[\left(\widehat{\Gg}^\alpha W_n(t_i)\right)^2\right]=\sum_{k=1}^{i}\left(\sqrt{\int_{t_{k-1}}^{t_k} g(t_i-s)^2\D s}\right)^2 = \int_0^{t_i} g(t_i-s)^2\D s=\EE\left[\left(\Gg^\alpha W(t_i)\right)^2\right],
$$
and by Cauchy-Schwarz we also have
\begin{align*}
\EE\left[\widehat{\Gg}^\alpha W_n(t_i)\widehat{\Gg}^\alpha W_n(t_j)\right]&= \sum_{k=1}^{t_i\wedge t_j}\sqrt{\int_{t_{k-1}}^{t_k} g(t_i-s)^2 \D s \int_{t_{k-1}}^{t_k} g(t_j-s)^2 \D s } \\ &\geq \int_0^{t_i\wedge t_j} g(t_i-s)g(t_j-s)\D s=\EE\left[\Gg^\alpha W(t_i)\Gg^\alpha W(t_j)\right].
\end{align*}
We then obtain
$\widetilde{\sigma}^2\leq  \EE\left[\left|\Gg^\alpha W(t_i)-\Gg^\alpha W(t_j)\right|^{2}\right]$.
Finally, $\widetilde{\Gg}^\alpha W_n(t_i)-\widetilde{\Gg}^\alpha W(t_j)$ is sub-Gaussian as a linear combination of sub-Gaussian random variables, 
and the Gaussian moment inequality~\cite[Theorem 2.1]{Boucheron13} with the variance estimate~$\widetilde{\sigma}^2$ yields
$$
\EE\left[\left|\widehat{\Gg}^\alpha W_n(t_i)-\widehat{\Gg}^\alpha W(t_j)\right|^{2p}\right]\leq  \EE\left[\left|\Gg^\alpha W(t_i)-\Gg^\alpha W(t_j)\right|^{2p}\right]\leq \widetilde{C} |t_i-t_j|^{2p\alpha+p}.
$$
\end{proof}

\subsubsection{Reducing variance} 
As Bayer, Friz and Gatheral~\cite{BFG16} and Bennedsen, Lunde and Pakkanen~\cite{BLP17} pointed out, 
a major drawback in simulating rough volatility models is the very high variance of the estimators,
so that a large number of simulations are needed to produce a decent price estimate. 
Nevertheless, the rDonsker scheme admits a very simple conditional expectation technique 
which reduces both memory requirements and variance while also admitting antithetic variates. 
This approach is best suited for calibrating European type options.
We consider $\Ff^{B}_t=\sigma(B_s:s\leq t)$ and $\Ff^{W}_t=\sigma(W_s:s\leq t)$ 
the natural filtrations generated by the Brownian motions~$B$ and~$W$. 
In particular the conditional variance process $V_t|\Ff^{W}_t$ is deterministic. 
As discussed by Romano and Touzi~\cite{RT97},
and recently adapted to the rBergomi case by McCrickerd and Pakkanen~\cite{MP18},
we can decompose the stock price process as 
$$
\E^{X_t} = \Ee\left(\rho\int_0^t\sqrt{\Phi\left(\Gg^\alpha Y\right)(s)}\D W_s\right)
\Ee\left(\sqrt{1-\rho^2}\int_0^t\sqrt{\Phi\left(\Gg^\alpha Y\right)(s)}\D W^\bot_s\right)
 =: \E^{X^{||}_t}\E^{X^\bot_t},
$$
and notice that
$$
X_t\vert(\Ff^{W}_t\wedge \Ff^{B}_0)
\sim \Nn\left(X^{||}_t - (1-\rho^2)\int_{0}^{t}\Phi\left(\Gg^\alpha Y\right)(s)\D s, 
(1-\rho^2) \int_0^t\Phi\left(\Gg^\alpha Y\right)(s)\D s\right).
$$
Thus $\exp(X_t)$ becomes log-normal and the Black-Scholes closed-form formulae are valid here 
(European, Barrier options, maximum, etc.). The advantage of this approach is that the orthogonal Brownian motion~$W^\bot$ 
is completely unnecessary for the simulation, 
hence the generation of random numbers is reduced to a half, yielding proportional memory saving. 
Not only this, but this simple trick also reduces the variance of the Monte-Carlo estimate, 
hence fewer simulations are needed to obtain the same precision. 
We present a simple algorithm to implement the rDonsker with conditional expectation 
and assuming that $Y=W$.
\begin{algorithm}[Simulation of rough volatility models with Brownian drivers]\label{algo:rDonsker variance reduction}
Consider the equidistant grid~$\Tt$.
\begin{enumerate}
\item Draw a random matrix
$\{\zeta_{j,i}\}_{\substack{j=1,\ldots,M/2\\i=1,\ldots,n}}$  
with unit variance, 
and create antithetic variates $\{-\zeta_{j,i}\}_{\substack{j=1,\ldots,M/2\\i=1,\ldots,n}}$;
\item \label{step 2}Simulate~$M$ paths of the fractional driving process~$\Gg^\alpha W$ 
using discrete convolution (see Algorithm~\ref{algo: convolution for rough vol} 
in Appendix~\ref{sec: convolution} for details in the implementation):
$$
(\Gg^\alpha W)^j(\Tt)=\sqrt{\frac{T}{n}}(\gr\ast\zeta_j),\quad j=1,\ldots,M,
$$
and store in memory
$\displaystyle (1-\rho^2)\int_0^T(\Gg^\alpha W)^j(s)\D s
\approx (1-\rho^2)\frac{T}{n}\sum_{k=0}^{n-1} (\Gg^\alpha W)^j(t_k)
 =: \Sigma^j$ for each $j=1,\ldots,M$;
\item use the forward Euler scheme to simulate the log-stock process,
for each $i=1,\ldots,n$, $j=1,\ldots,M$, as
$$
X^j(t_i) = X^j(t_{i-1})
 - \frac{\rho^2}{2}\frac{T}{n}\Phi\left(\Gg^\alpha W\right)^j(t_{i-1})
  + \rho\sqrt{\frac{T}{n}}\sqrt{\Phi\left(\Gg^\alpha W\right)^j(t_{i-1})}\zeta_{j,i};
$$
\item Finally, we may compute any option using the Black-Scholes formula. For instance a Call option with strike $K$  and maturity $T\in\TT$ would be given by
$C^j(K)=\exp(X^j(T))\mathcal{N}(d^j_1)-K\mathcal{N}(d^j_2)$ for $j=1,\ldots,M$,
where $\Sigma^j=Var(X^j(T))$,
$d^j_1:=\frac{1}{\sqrt{\Sigma^j}}(X^j(T)-\log(K)+\frac{1}{2}\Sigma^j)$ 
and $d^j_2=d^j_1-\sqrt{\Sigma^j}$.
Thus, the output of the model would be
$C(K)=\frac{1}{M}\sum_{k=1}^{M}C^j(K)$.
\end{enumerate}
\end{algorithm}
The algorithm is easily adapted to general diffusions~$Y$ as drivers of the volatility 
(Algorithm~\ref{algo:rDonsker}\eqref{algo: step2}).
Algorithm~\ref{algo:rDonsker} is obviously faster than~\ref{algo:rDonsker variance reduction},
especially when using control variates. 
Nevertheless, with the same number of paths, Algorithm~\ref{algo:rDonsker variance reduction} 
remarkably reduces the Monte-Carlo variance, 
meaning in turn that fewer simulations are needed, making it very competitive for calibration.

\subsection{Numerical example: Rough Bergomi model}
Figures~\ref{fig: Monte_carlo_Cholesky}-\ref{fig: Monte_carlo_rDonsker_left} perform a numerical analysis of the Monte Carlo convergence as a function of $n$. 
We observe that the lower the~$H$, the larger~$n$ needs to be to achieve convergence. 
However, we also observe that for the Cholesky, rDonsker 
(naive and moment match) and Hybrid schemes and $H\geq 0.1$, 
with $N=252$ we already achieve a precision of 
order~$10^{-4}$, which is equivalent to a basis point in financial terms. 
For $H<0.1$ we might require~$n$ larger than~$252$, if precision is required beyond~$10^{-4}$. 
We also observe in Figure~\ref{fig: Monte_carlo_rDonsker_left} that the naive rDonsker approximation converges extremely slow for small $H$. Additionally, Figures~\ref{fig: Monte_carlo_error_H_1}-\ref{fig: Monte_carlo_error_H_40} measure the price estimations compared to the Cholesky method which is taken as benchmark. 
The Hybrid scheme tends to be closer to this benchmark 
especially for $H< 0.1$. 
When $H\geq 0.1$ for both the Hybrid scheme and rDonsker moment-match we observe an error less than~$10^{-4}$ for $n\geq 252$. 
It is noteworthy to mention that the naive rDonsker scheme has substantially worse convergence (at least an order of magnitude) than the other methods.  We note that the black lines in all figures represent the $99\%$ Monte Carlo standard deviations, hence errors below that threshold should be interpreted as noise.

\begin{figure}[h!]
\centering
\includegraphics[scale=0.5]{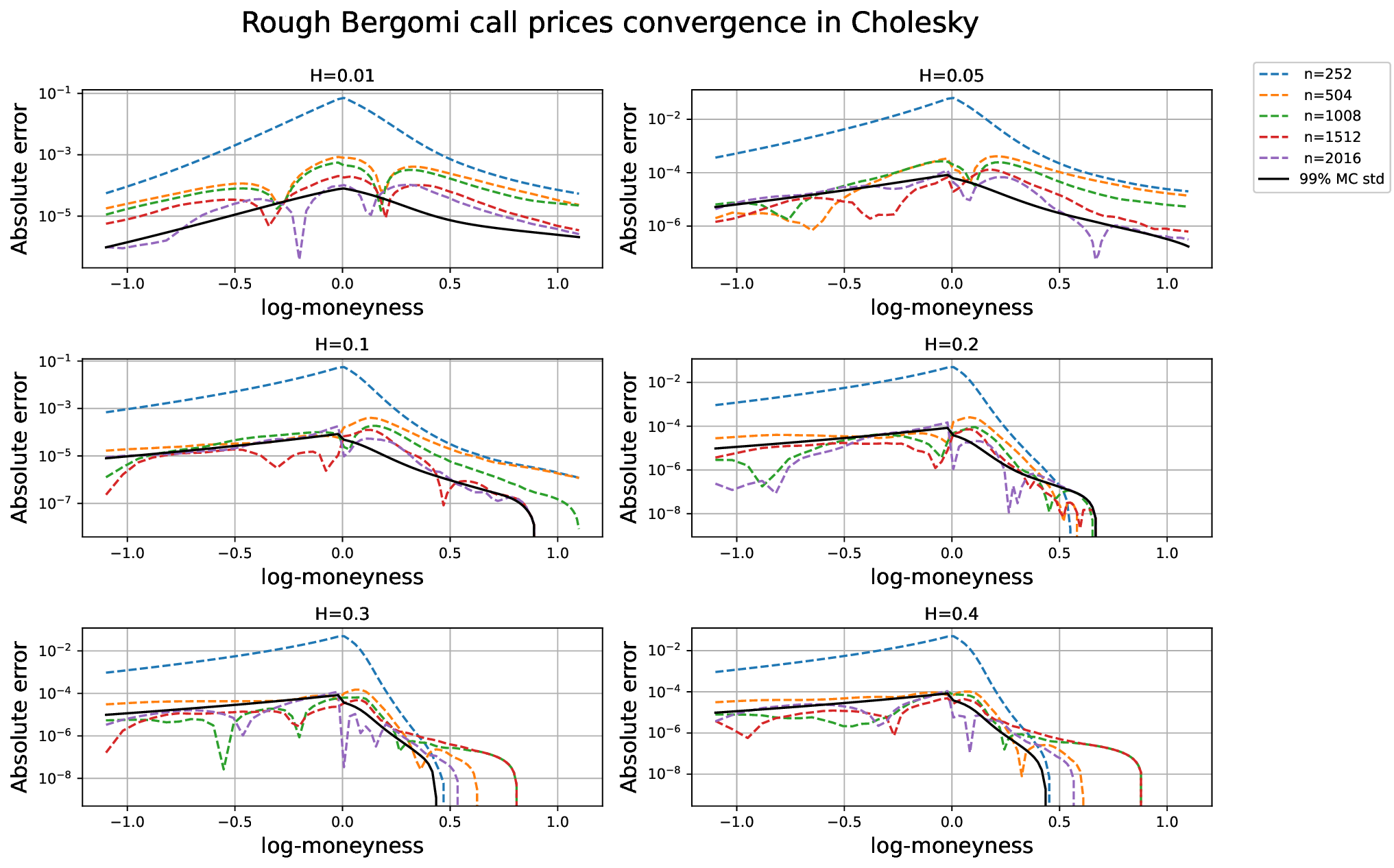}
\caption{Rough Bergomi Call option price convergence using Cholesky method with $\xi_0=0.04,\nu=2.3, \rho=-0.9, S_0=1, T=1$ with $2\cdot10^6$ simulations and antithetic variates.  Absolute error represents the difference between subsequent approximations, where $n$ represents the time grid size. For $n=252$ the previous discretisation is $n=126$.}
\label{fig: Monte_carlo_Cholesky}
\end{figure}

\begin{figure}[h!]
\centering
\includegraphics[scale=0.5]{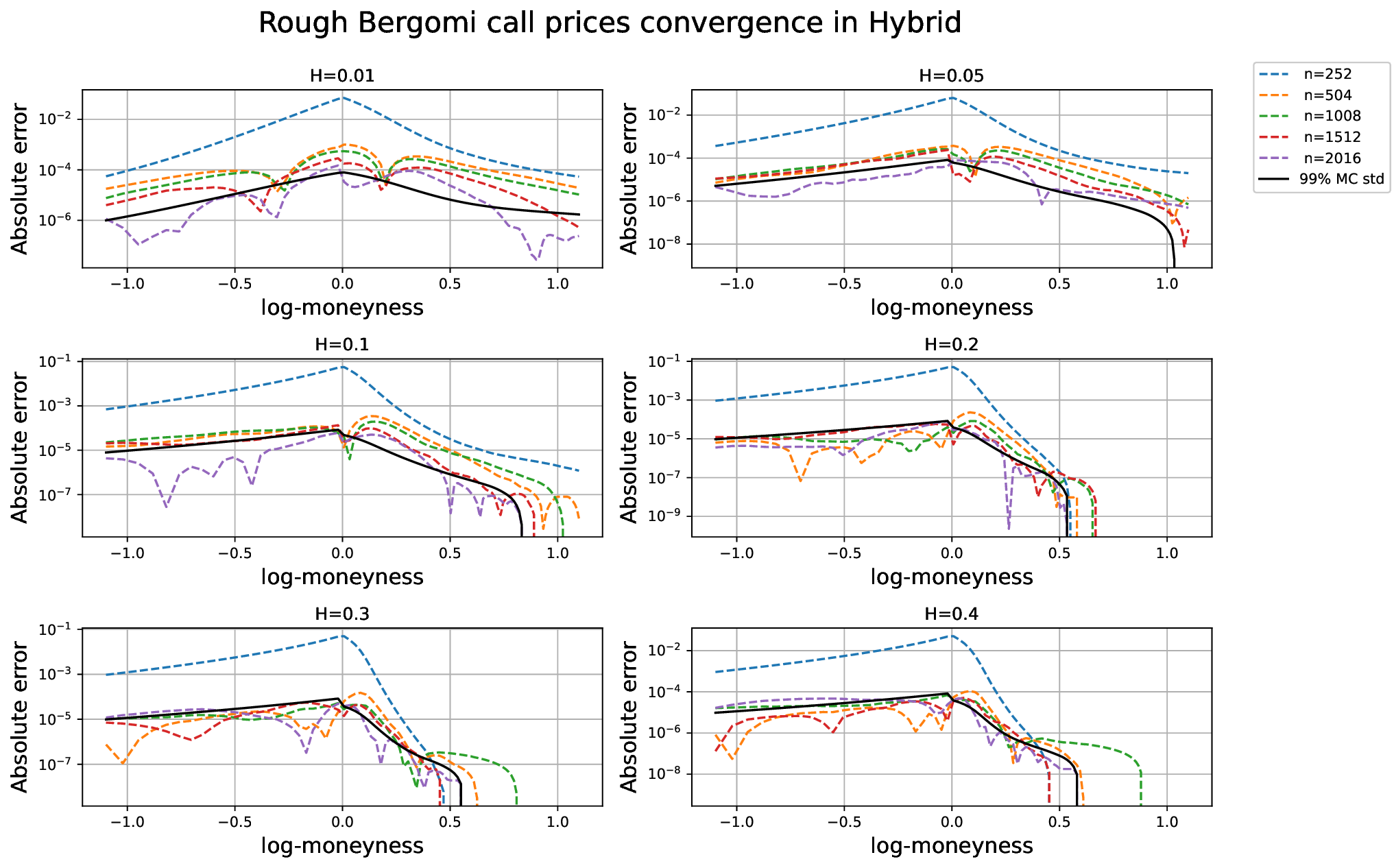}
\caption{Rough Bergomi Call option price convergence using Cholesky method with $\xi_0=0.04,\nu=2.3, \rho=-0.9, S_0=1, T=1$ with $2\cdot10^6$ simulations and antithetic variates.  Absolute error represents the difference between subsequent approximations, where $n$ represents the time grid size. For $n=252$ the previous discretisation is $n=126$.}
\label{fig: Monte_carlo_Hybrid}
\end{figure}

\begin{figure}[h!]
\centering
\includegraphics[scale=0.5]{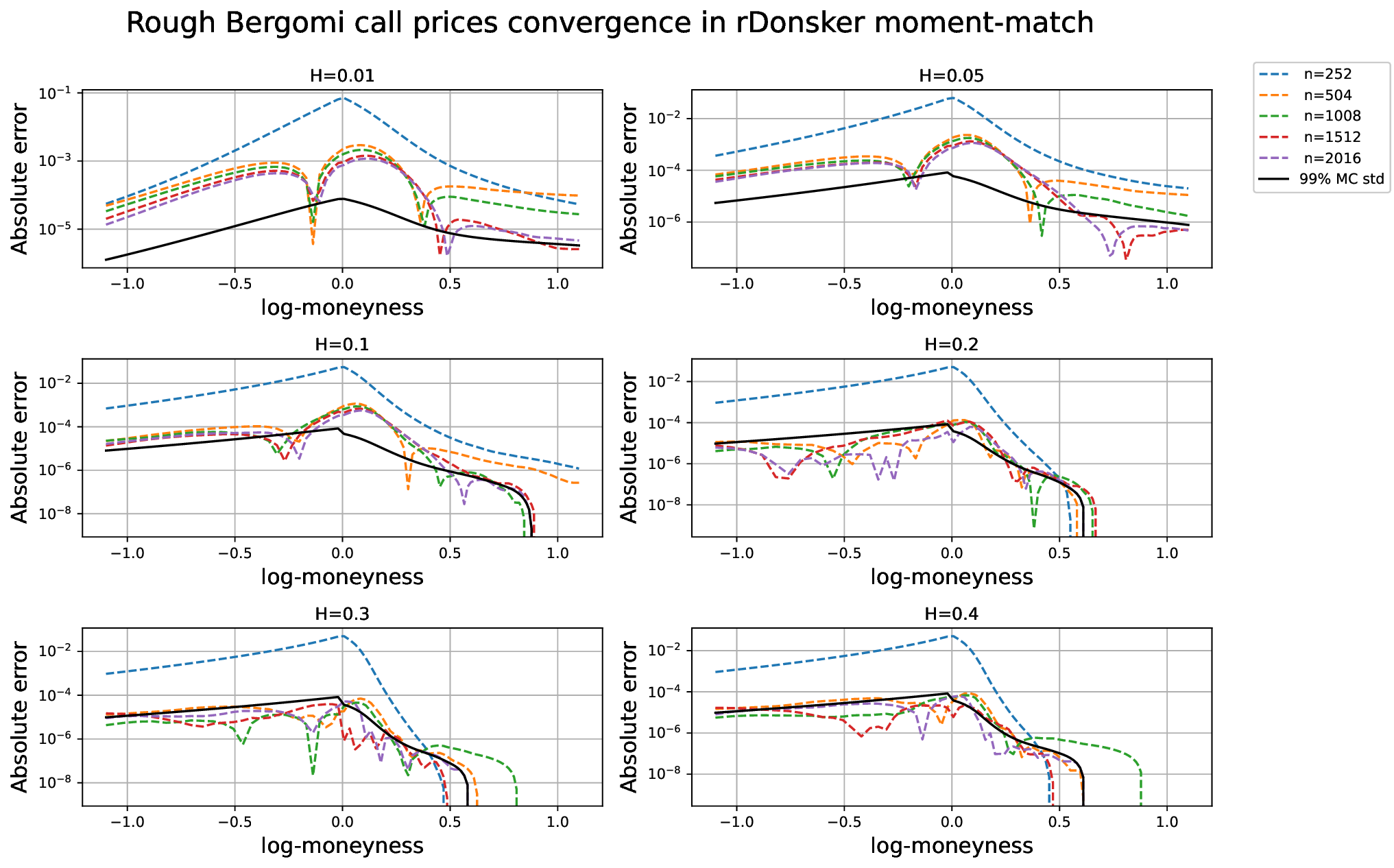}
\caption{Rough Bergomi Call option price convergence using rDonsker with moment-matching and 
$\xi_0=0.04,\nu=2.3, \rho=-0.9, S_0=1, T=1$ with $2\cdot10^6$ simulations and antithetic variates.  Absolute error represents the difference between subsequent approximations, 
with~$n$ the time grid size. 
For $n=252$ the previous discretisation is $n=126$.}
\label{fig: Monte_carlo_rDonsker}
\end{figure}

\begin{figure}[h!]
\centering
\includegraphics[scale=0.5]{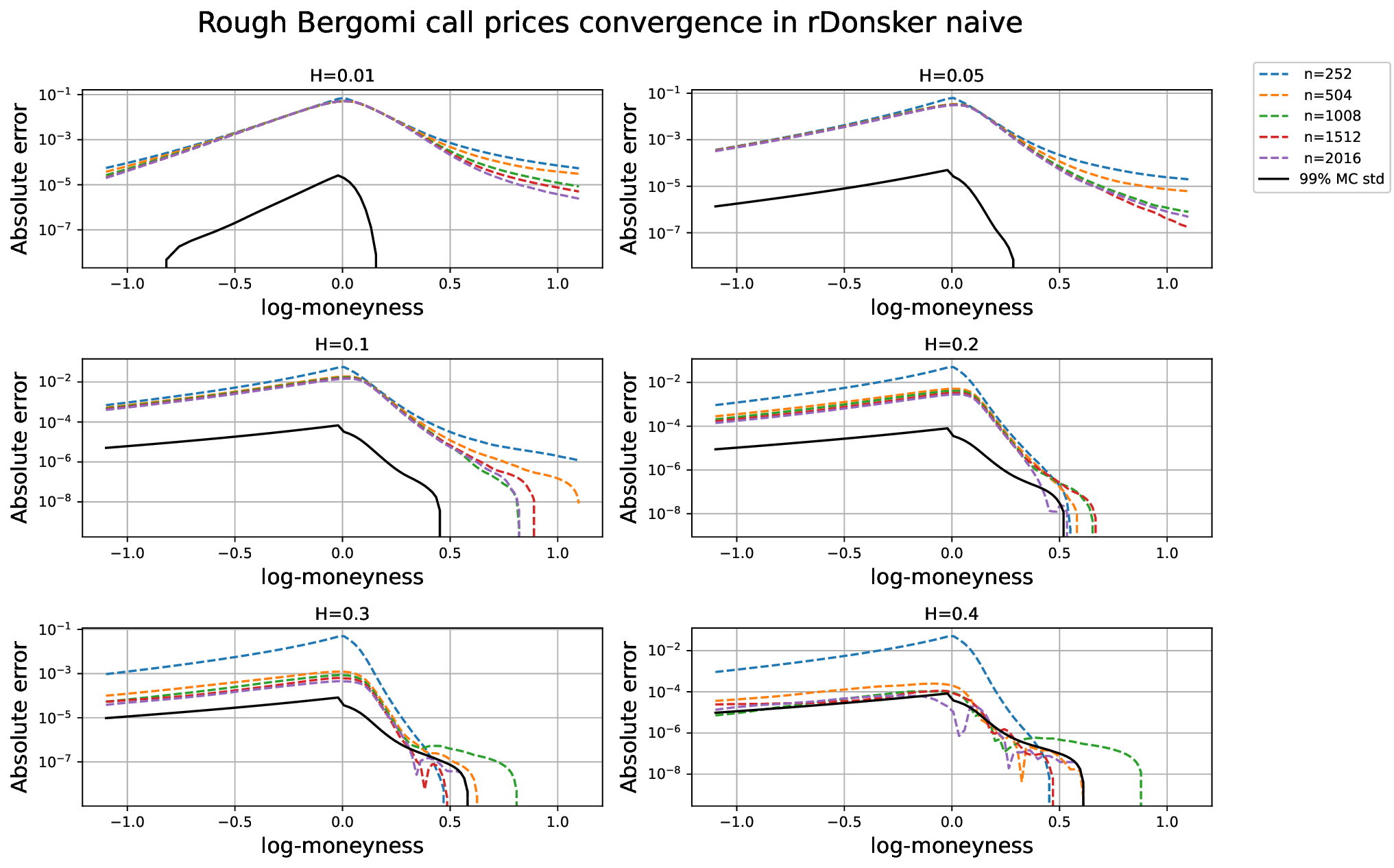}
\caption{Rough Bergomi Call option price convergence using rDonsker method with left point Euler and $\xi_0=0.04,\nu=2.3, \rho=-0.9, S_0=1, T=1$ with $2\cdot10^6$ simulations and antithetic variates.  Absolute error represents the difference between subsequent approximations, where $n$ represents the time grid size. For $n=252$ the previous discretisation is $n=126$.}
\label{fig: Monte_carlo_rDonsker_left}
\end{figure}

\begin{figure}[h!]
\centering
\includegraphics[scale=0.5]{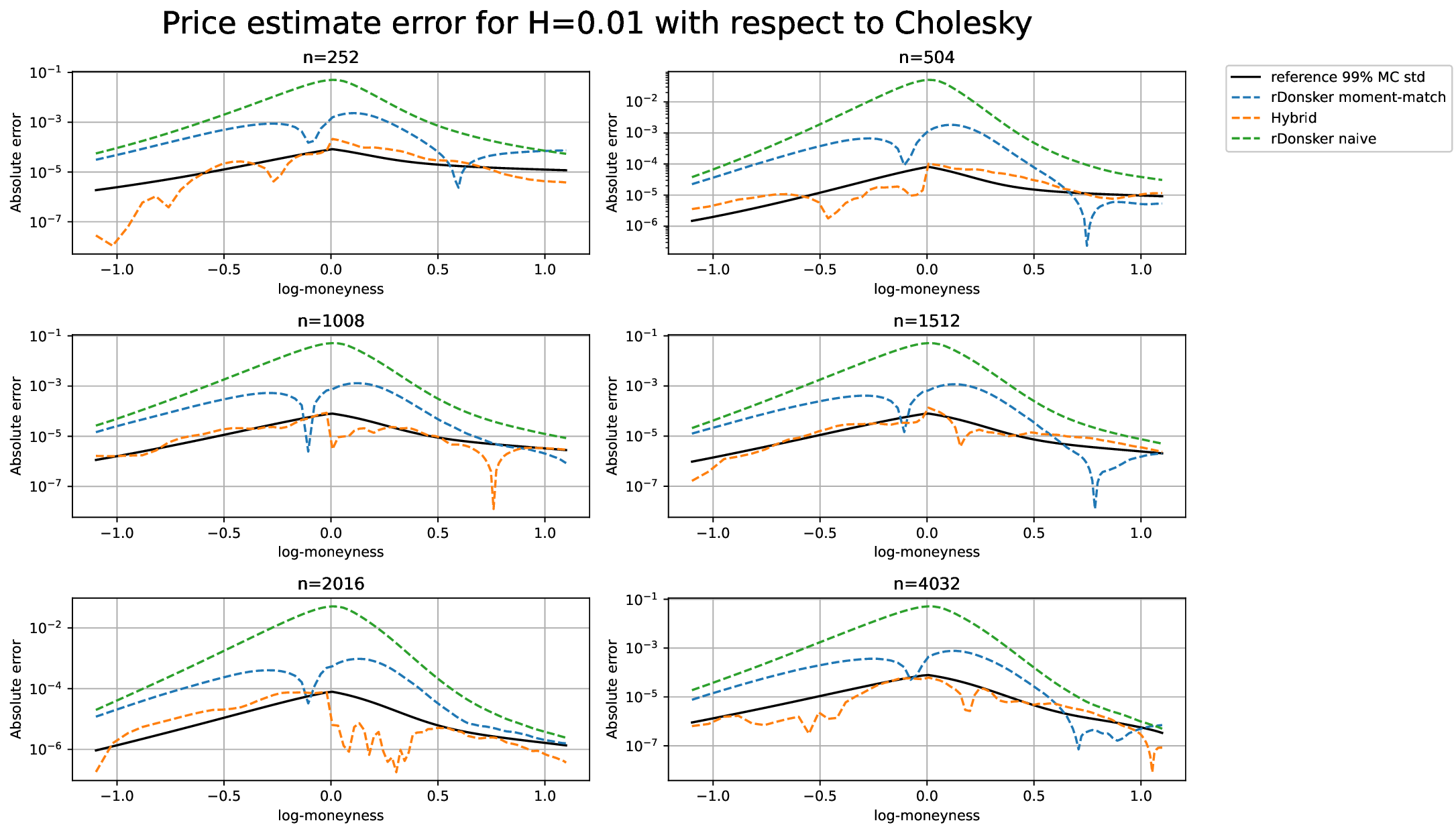}
\caption{Rough Bergomi Call option price comparison with $H=0.01,\xi_0=0.04,\nu=2.3, \rho=-0.9, S_0=1, T=1$ with $2\cdot10^6$ simulations and antithetic variates.  Absolute error represents the difference in price between different simulation schemes.}
\label{fig: Monte_carlo_error_H_1}
\end{figure}

\begin{figure}[h!]
\centering
\includegraphics[scale=0.5]{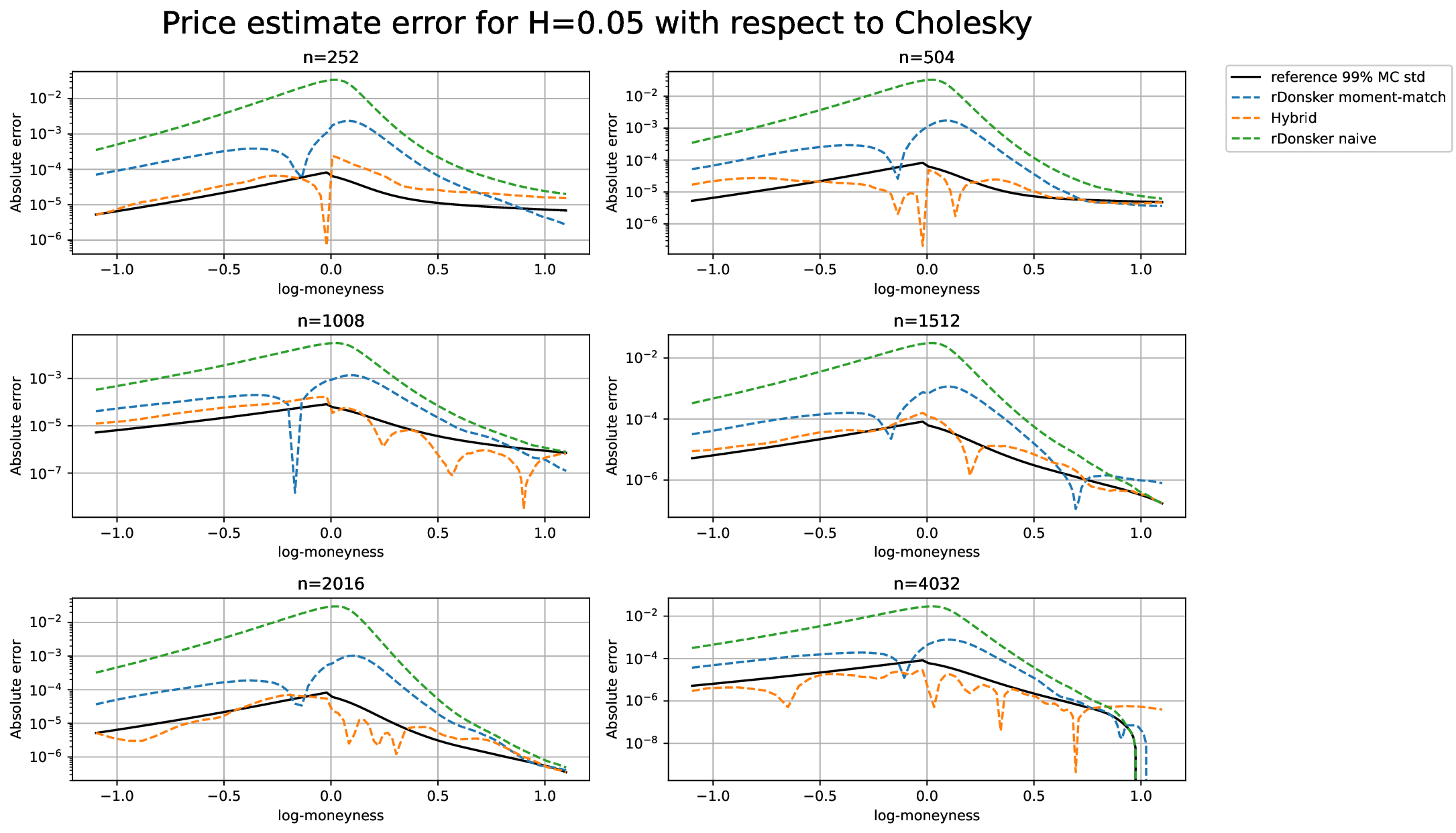}
\caption{Rough Bergomi Call option price comparison with $H=0.05,\xi_0=0.04,\nu=2.3, \rho=-0.9, S_0=1, T=1$ with $2\cdot10^6$ simulations and antithetic variates.  Absolute error represents the difference in price between different simulation schemes.}
\label{fig: Monte_carlo_error_H_5}
\end{figure}

\begin{figure}[h!]
\centering
\includegraphics[scale=0.5]{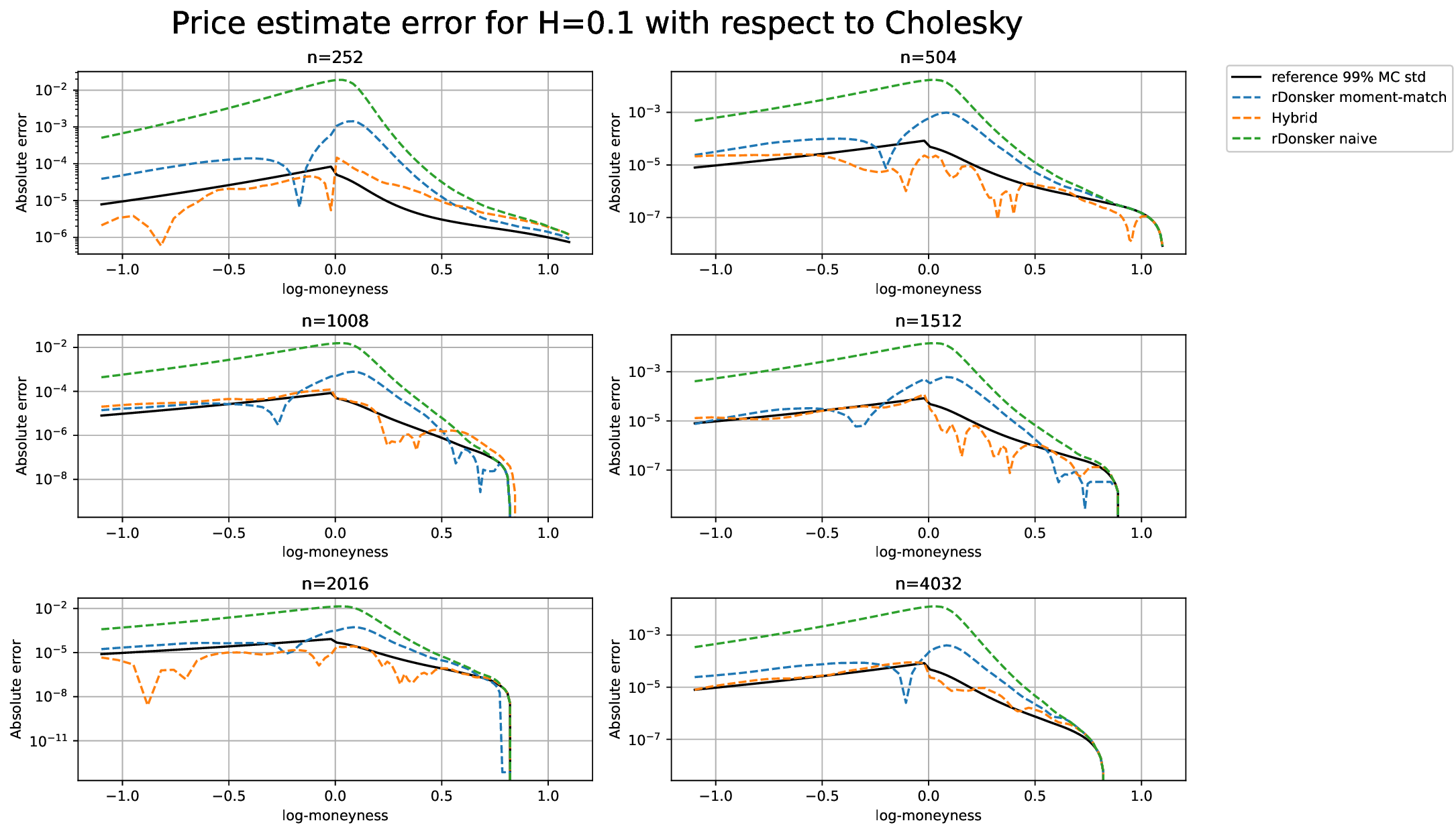}
\caption{Rough Bergomi Call option price comparison with $H=0.10,\xi_0=0.04,\nu=2.3, \rho=-0.9, S_0=1, T=1$ with $2\cdot10^6$ simulations and antithetic variates.  Absolute error represents the difference in price between different simulation schemes.}
\label{fig: Monte_carlo_error_H_10}
\end{figure}

\begin{figure}[h!]
\centering
\includegraphics[scale=0.5]{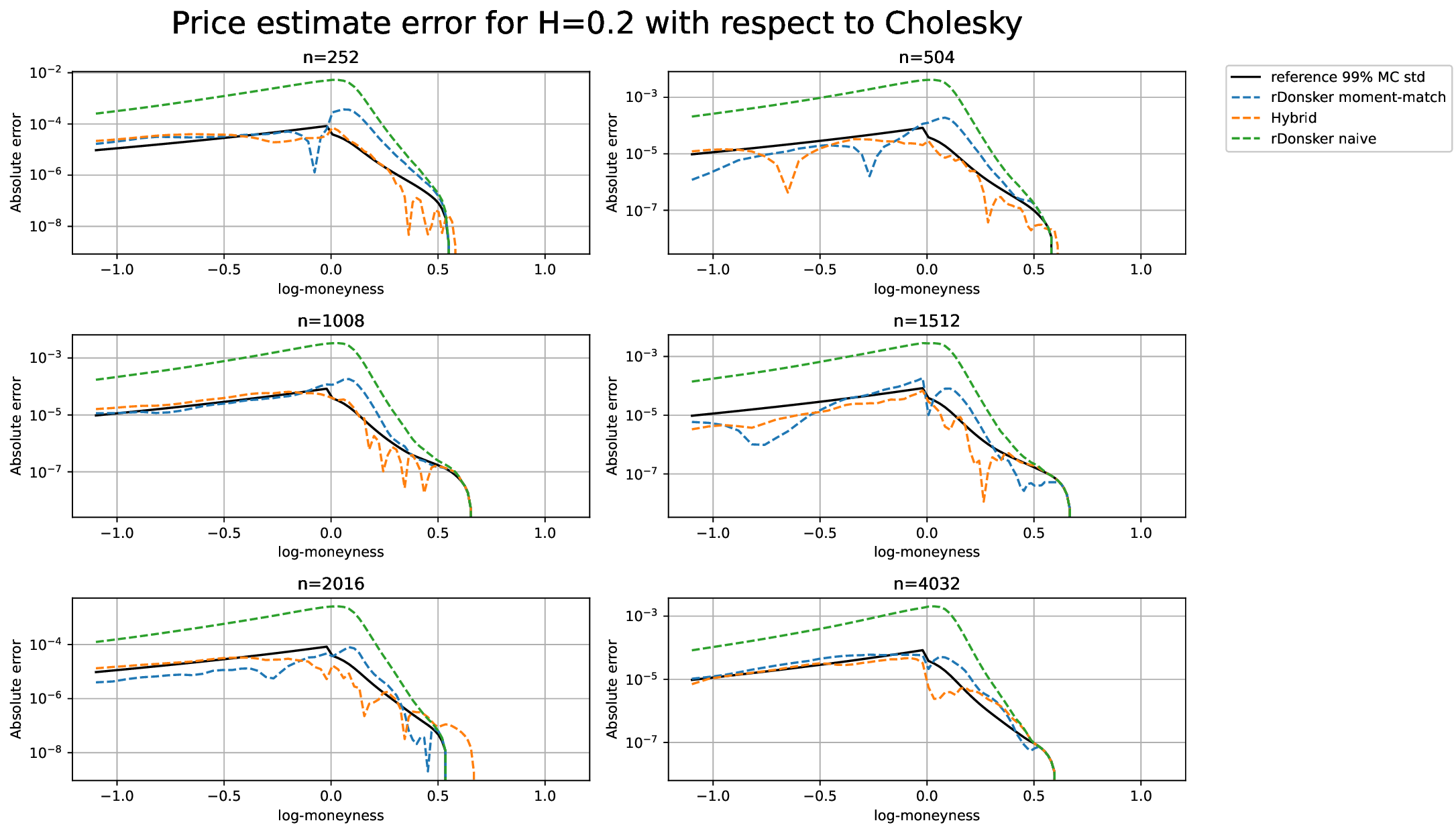}
\caption{Rough Bergomi Call option price comparison with $H=0.2,\xi_0=0.04,\nu=2.3, \rho=-0.9, S_0=1, T=1$ with $2\cdot10^6$ simulations and antithetic variates.  Absolute error represents the difference in price between different simulation schemes.}
\label{fig: Monte_carlo_error_H_20}
\end{figure}

\begin{figure}[h!]
\centering
\includegraphics[scale=0.5]{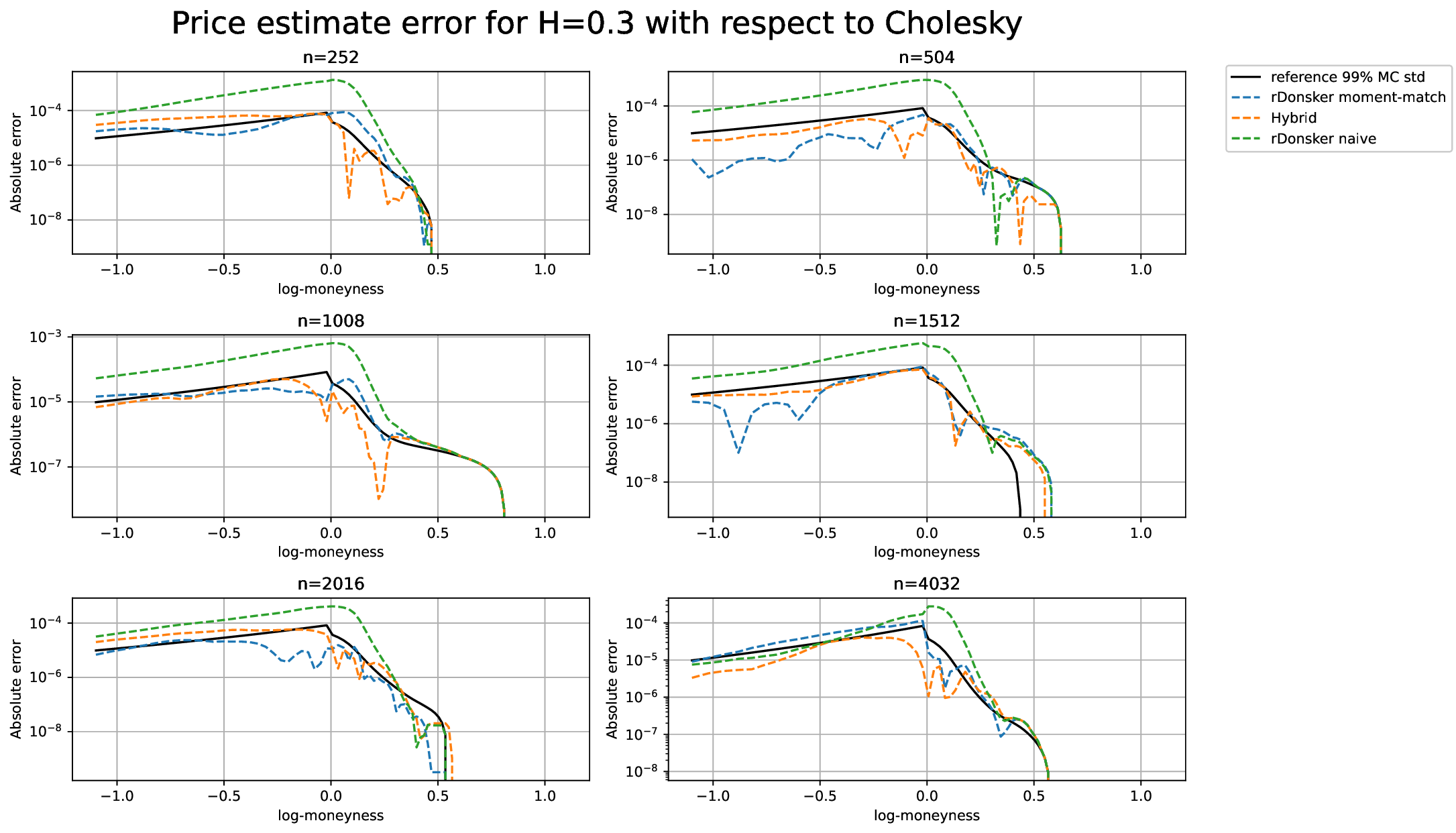}
\caption{Rough Bergomi Call option price comparison with $H=0.3,\xi_0=0.04,\nu=2.3, \rho=-0.9, S_0=1, T=1$ with $2\cdot10^6$ simulations and antithetic variates.  Absolute error represents the difference in price between different simulation schemes.}
\label{fig: Monte_carlo_error_H_30}
\end{figure}

\begin{figure}[h!]
\centering
\includegraphics[scale=0.5]{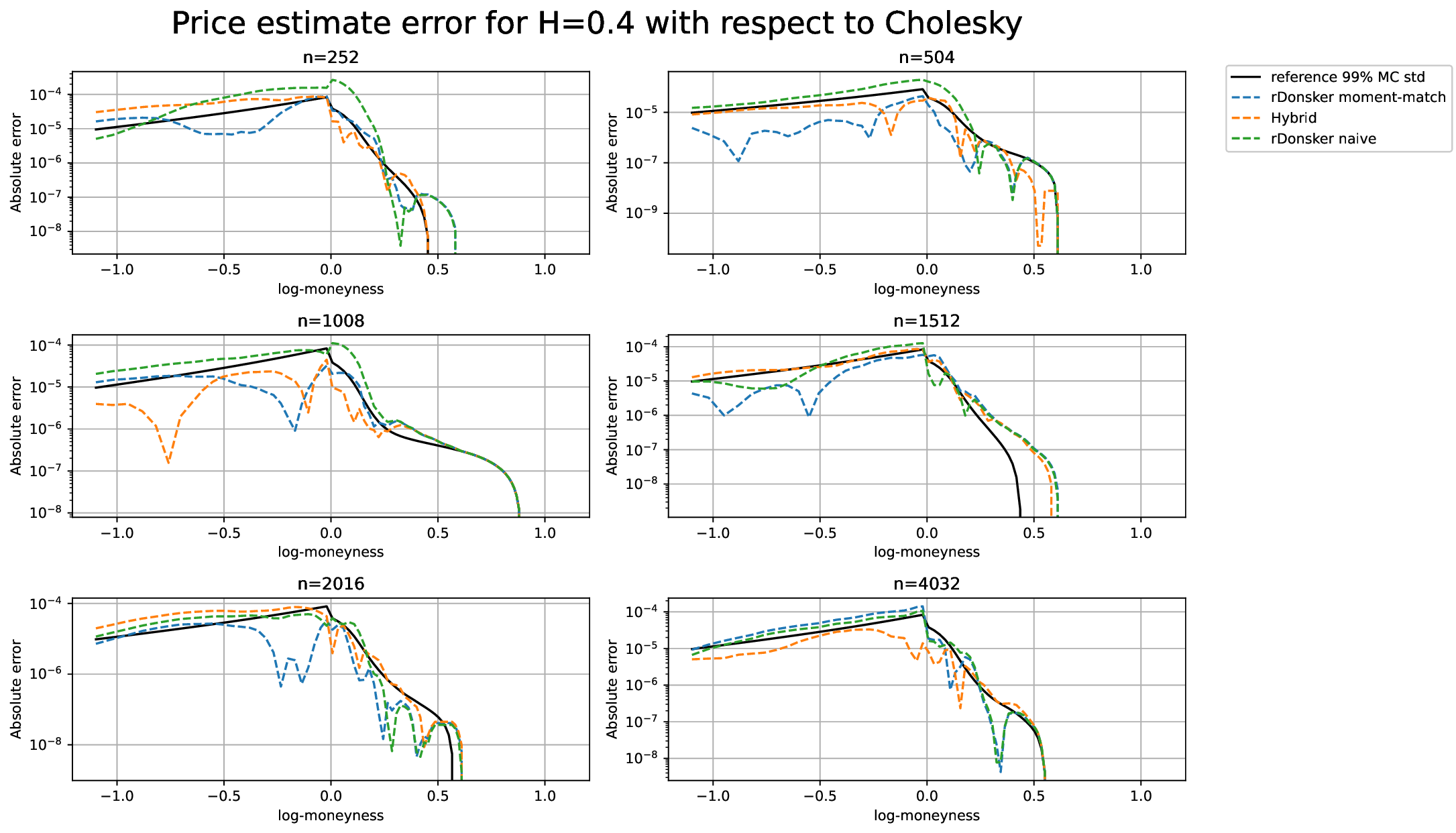}
\caption{Rough Bergomi Call option price comparison with $H=0.4,\xi_0=0.04,\nu=2.3, \rho=-0.9, S_0=1, T=1$ with $2\cdot10^6$ simulations and antithetic variates.  Absolute error represents the difference in price between different simulation schemes.}
\label{fig: Monte_carlo_error_H_40}
\end{figure}

\subsection{Speed benchmark against Markovian stochastic volatility models}
In this section we benchmark the speed of the rDonsker scheme against the Hybrid scheme and a classical Markovian stochastic volatility model using~$10^5$ simulations and averaging the speeds over~$10$ trials. 
For the former we simulate the rBergomi model~\cite{BFG16}, whereas for the latter we use the classical Bergomi~\cite{Bergomi05} model using a forward Euler scheme in both volatility and stock price. 
All three schemes are implemented in \texttt{Cython} to make the comparison fair, 
and to obtain speeds comparable to~\texttt{C++}. 
Figure~\ref{fig: speeds} shows that rDonsker is about twice slower than the Markovian case whereas the Hybrid scheme is approximately 2.5 times slower, which is expected from the complexities of both schemes. 
However, it is remarkable that the $\Oo(n\log n)$ complexity of the FFT stays almost constant with the grid size $n$ and the computational time grows almost linearly as in the Markovian case. 
We presume that this is the case since $n\ll 10000$ is relatively small. 
Figure~\ref{fig: speeds} also shows that rough volatility models can be implemented very efficiently and are not particularly slower than classical stochastic volatility models.

\begin{figure}[h!]
\centering
\includegraphics[scale=0.5]{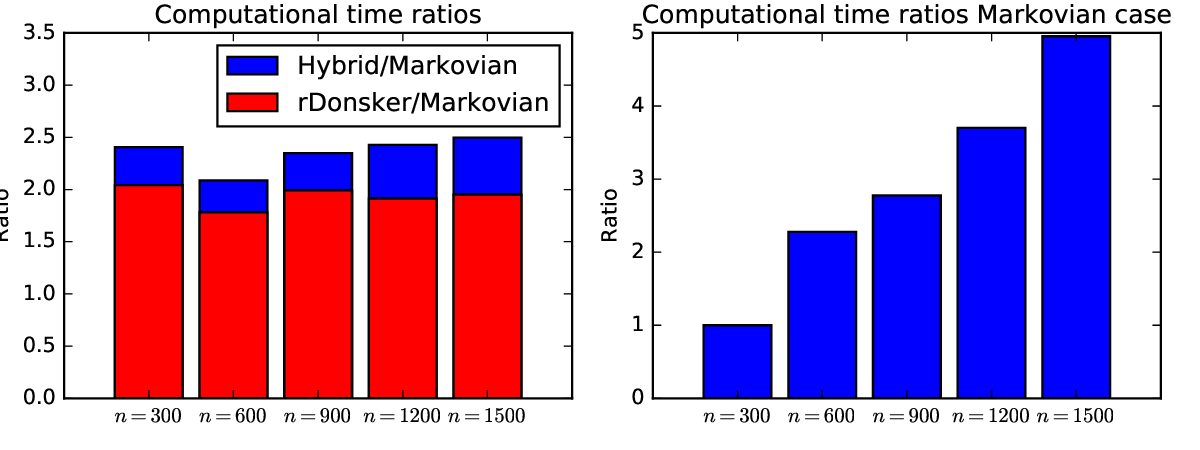}
\caption{Computational time benchmark using Hybrid scheme, rDonsker and Markovian (forward Euler) for different grid sizes~$n$.}
\label{fig: speeds}
\end{figure}

\subsection{Implementation guidelines and conclusion}
The numerical analysis above suggests some guidelines to implement rough volatility models driven by~$\TBSS$ processes
of the form $\Gg^{H-1/2} Y$, for some It\^o diffusion~$Y$:
\begin{table}[ht]
\centering 
\begin{tabular}{|c | c | c|} 
\hline
$H>0.1$ & $H\in[0.05, 0.1]$ & $H<0.05$ \\ [0.5ex]
\hline 
rDonsker & choice depends on error sensitivity & Hybrid scheme\\
\hline 
\end{tabular}
\end{table}

Regarding empirical estimates, Gatheral, Jaisson and Rosenbaum~\cite{GJR18} suggest that $H\approx 0.15$. 
Bennedsen, Lunde and Pakkanen~\cite{BLP16} give an exhaustive analysis of more than $2000$ equities 
for which $H\in[0.05,0.2]$. 
On the pricing side, Bayer, Friz and Gatheral~\cite{BFG16} and Jacquier, Martini and Muguruza~\cite{JMM18} found that calibration routines yield $H\in[0.05,0.10]$. 
Finally, Livieri, Mouti, Pallavicini and Rosenbaum~\cite{LMP+18} found evidence in options data that 
$H \approx 0.3$. 
Despite the diverse ranges found so far, there is a common agreement that $H<1/2$.
\begin{remark}
The rough Heston model presented by Guennoun, Jacquier, Roome and Shi~\cite{GJR+18} 
is out of the scope of the Hybrid scheme. 
Moreover, any process of the form~$\Gg^\alpha Y$, for some It\^o diffusion~$Y$ 
under Assumptions~\ref{assu:CoeffSDE} is, in general, out of the scope of the Hybrid scheme. 
This only leaves the choice of using  the rDonsker scheme, 
for which reasonable accuracy is obtained at least for H\"older regularities greater than~$0.05$.
\end{remark}

\subsection{Bushy trees and binomial markets}
Binomial trees have attracted a lot of attention from both academics and practitioners,
as their apparent simplicity provides easy intuition about the dynamics of a given asset. 
Not only this, but they are by construction arbitrage free and allow to price path-dependent options, 
together with their hedging strategy. 
In particular, early exercise options, in particular Bermudan or American options, 
are usually priced using trees, as opposed to Monte-Carlo methods.
The convergence stated in Theorem~\ref{thm:MainThm} lays the theoretical foundations 
to construct fractional binomial trees 
(note that Bernoulli random variables satisfy the conditions of the theorem). 
Figure~\ref{bino_tree_H_75} already showed binomial trees 
for fractional Brownian motion, 
but we ultimately need trees describing the dynamics of the stock price.

\subsubsection{A binary market}\label{sec: binary market}
We invoke Theorem~\ref{thm:MainThm} with the independent sequences
$\{\zeta_i\}_{i=1}^n$, $\{\zeta_i^\bot\}_{i=1}^n$ such that 
$\PP(\zeta_i=1)=\PP(\zeta^\bot_i=1)=\PP(\zeta_i=1)=\PP(\zeta^\bot_i=-1)=\frac{1}{2}$ for all~$i$. 
We further define, on~$\Tt$, for any $i=1, \ldots, n$,
\begin{align*}
B_n(t_{i})
 & = \sqrt{\frac{T}{n}}\sum_{k=1}^{i}\left(\rho\zeta_k+\rrho\zeta^\bot_k\right),\\
 Y_n(t_{i})
 & = \frac{T}{n}\sum_{k=1}^{i}b\left(Y_n(t_{k-1})\right)
  + \sqrt{\frac{T}{n}}\sum_{k=1}^{i}\sigma\left(Y_n(t_{k-1})\right)\zeta_k,
\end{align*}
the approximating sequences to~$B$ and~$Y$ in~\eqref{eq:System}. 
The approximation for~$X$ is then given by
$$
X_n(t_{i}) = X_n(t_{i-1})
  - \frac{1}{2}\frac{T}{n}\sum_{k=1}^{i}\Phi\left(\Gg^\alpha Y_n\right)(t_{k})
  + \sqrt{\frac{T}{n}}\sum_{k=1}^{i}\sqrt{\Phi\left(\Gg^\alpha Y_n\right)(t_{k})}\left(\rho\zeta_k+\rrho\zeta^\bot_k\right).
$$
In order to construct the tree we have to consider all possible permutations 
of the random vectors~$\{\zeta_i\}$ and~$\{\zeta^\bot_i\}$. 
Since each random variable only takes two values, this adds up to~$4^n$ possible combinations, 
hence the `bushy tree' terminology. 
When $\rho\in\{-1,1\}$, the magnitude is reduced to~$2^n$. 

\subsection{American options in rough volatility models}
There is so far no available scheme for American options (or any early-exercise options for that matter)
under rough volatility models, 
but the fractional trees constructed above provide a framework to do so. 
In the Black-Scholes model, American options can be priced using binomial trees by backward induction.
A key ingredient is the Snell envelope~\cite{Snell52} 
and the following representation by El Karoui~\cite{ElKaroui81}
($\TTw$ denotes the set of stopping times with values in~$\TT$):
\begin{definition}
Let $(X_t)_{t\in\TT}$ be an~$(\Ff_t)_{t\in\TT}$ adapted process, and~$\tau\in\TTw$.
The Snell envelope~$\Jj$ of~$X$ is defined as 
$\Jj(X)(t) := \esssup_{\tau\in\TTw} \EE(X_\tau|\Ff_t)$ for all $t\in\TT$.
\end{definition}
In plain words, the Snell envelope of~$X$ is the smallest supermartingale that dominates it.
Strictly speaking, it is necessary for~$X_\tau$ to be uniformly integrable for any $\tau\in\TTw$.
Following~\cite{Karatzas88}, an American option is nothing else than 
the smallest supermartingale dominating its European counterpart:
\begin{definition}
Let $C^e_t(k,T)$ and $P^e_t(k,T)$ denote European Call and Put prices at time~$t$, 
with log-strike~$k$ and maturity~$T$. 
Then the American counterparts, $C^a_t(k,T)$ and $P^a_t(k,T)$, are given by
$$
C^a_t(k,T) = \Jj(C^e(k,T))(t)
\qquad\text{and}\qquad
P^a_t(k,T) = \Jj(P^e(k,T))(t).
$$
\end{definition}
Preservation of weak convergence under the Snell envelope map 
is due to Mulinacci and Pratelli~\cite{MP98}, 
who proved that convergence takes place in the Skorokhod topology only if the Snell envelope is continuous. 
In our setting, the scheme for American options is fully justified by the following theorem:
\begin{theorem}\label{thm:JConv}
For~$V$ in~\eqref{eq:System}, 
if $\E^X$ is a true martingale
then $\left(\Jj(\E^{X_n})\right)_{n\geq 1}$
converges weakly to~$\Jj\left(\E^X\right)$ in the Skorokhod topology $(\Dd(\TT),d_{\Dd})$.
\end{theorem}
\begin{proof}
Since the sequence $(X_n)_{n\geq 1}$ converges weakly to~$X$ in $(\Dd(\TT),d_{\Dd})$,
for $X$ in~\eqref{eq:System}, the theorem follows from the Continuous Mapping Theorem
if we can show that~$\Jj$ is continuous.
El Karoui proved in~\cite[Chapter 2.14]{ElKaroui81} that the Snell envelope of an optional process,
uniformly integrable for all stopping times $\tau\in\TTw$, is continuous.
To prove the proposition, we therefore only need to check uniform integrability of the stock price~$\E^{X}$.
As~$\TT$ is a finite time horizon, Doob's optimal stopping theorem for martingales gives 
$\E^{X_{t}}=\EE[\E^{X_1}|\mathcal{F}_{t}]$ for all $t\in\TT$, thus~$\E^X$ on~$\TT$ is uniformly integrable and the result follows.
\end{proof}
Mulinacci and Pratelli~\cite{MP98} also gave explicit conditions for weak convergence 
to be preserved in the Markovian case. 
It is trivial to see that the pricing of American options in the rough tree scheme coincides with the classical backward induction procedure. 
We consider continuously compounded interest rate~$r$ and dividend yield~$d$.
\begin{algorithm}[American options in rough volatility models]
On the equidistant grid~$\Tt$,
\begin{enumerate}
\item \label{s1}construct the binomial tree using the explicit construction in Section~\ref{sec: binary market} and obtain $\{S^j_t\}_{t\in\Tt,j=1,...,4^n}$;
\item the backward recursion for the American with exercise value $h(\cdot)$ is given by 
$\widetilde{h}_{t_N} := h(S_{t_N})$ and 
$$
\widetilde{h}_{t_i} := 
\E^{(d-r)/n}\EE\left[\widetilde{h}_{t_{i+1}}|\Ff_{t_i}\right] \vee h(S_{t_i}), 
\qquad \text{for } i=N-1,\ldots,0,
$$
where $\EE[\cdot|\Ff_{t_i}] = \frac{1}{4}\left(\widetilde{h}^{\mathrm{++}}_{t_{i+1}}+\widetilde{h}^{\mathrm{+-}}_{t_{i+1}}+\widetilde{h}^{\mathrm{-+}}_{t_{i+1}}+\widetilde{h}^{\mathrm{--}}_{t_{i+1}}\right)$ 
and
$\widetilde{h}_{t_i}^{\mathrm{\pm\pm}}$ represents the outcome $(\zeta_i,\zeta^\bot_i)=(\pm 1,\pm 1)$ for the driving binomials, following the construction in Section~\ref{sec: binary market}.
\item finally, $\widetilde{h}_{0}$ is the price of the American option at inception of the contract.
\end{enumerate}
\end{algorithm}
The main computational cost of the scheme is the construction of the tree in Step~\ref{s1}. 
Once the tree is constructed, computing American prices for different options is a fast routine.

\subsubsection{Numerical example: rough Bergomi model}
The rough Bergomi model satisfies the martingale property in Theorem~\ref{thm:JConv} (b) for $\rho\leq 0$ (see Gassiat~\cite{Gassiat19}). We construct a rough volatility tree for the rough Bergomi model~\cite{BFG16} and check the accuracy of the scheme. 
Figures~\ref{fig: rough_Bergomi_trees} and~\ref{fig: rough_Bergomi_trees2} 
show the fractional trees for different values of~$H$ and for $\rho\in\{-1,1\}$. 
Both pictures show a markedly different behaviour, 
but as a common property we observe that as~$H$ tends to~$1/2$, 
the tree structure somehow becomes simpler. 

\begin{figure}[h!]
\centering
\includegraphics[scale=0.5]{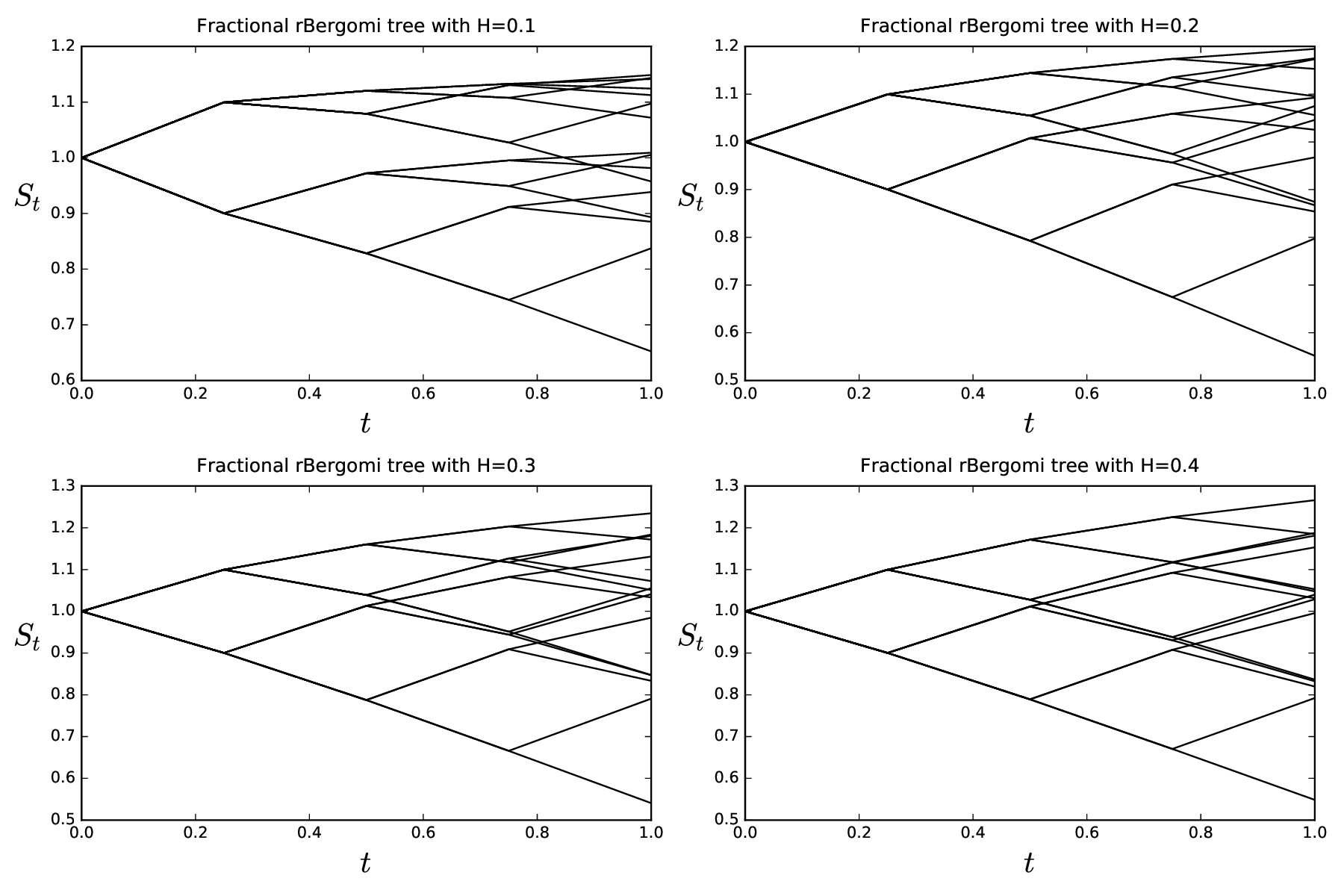}
\caption{rBergomi trees for different values of $H$, $(\nu, \rho, \xi_0) = (1,-1,0.04)$ 
with~$5$ time steps.}
\label{fig: rough_Bergomi_trees}
\end{figure}

\begin{figure}[h!]
\centering
\includegraphics[scale=0.5]{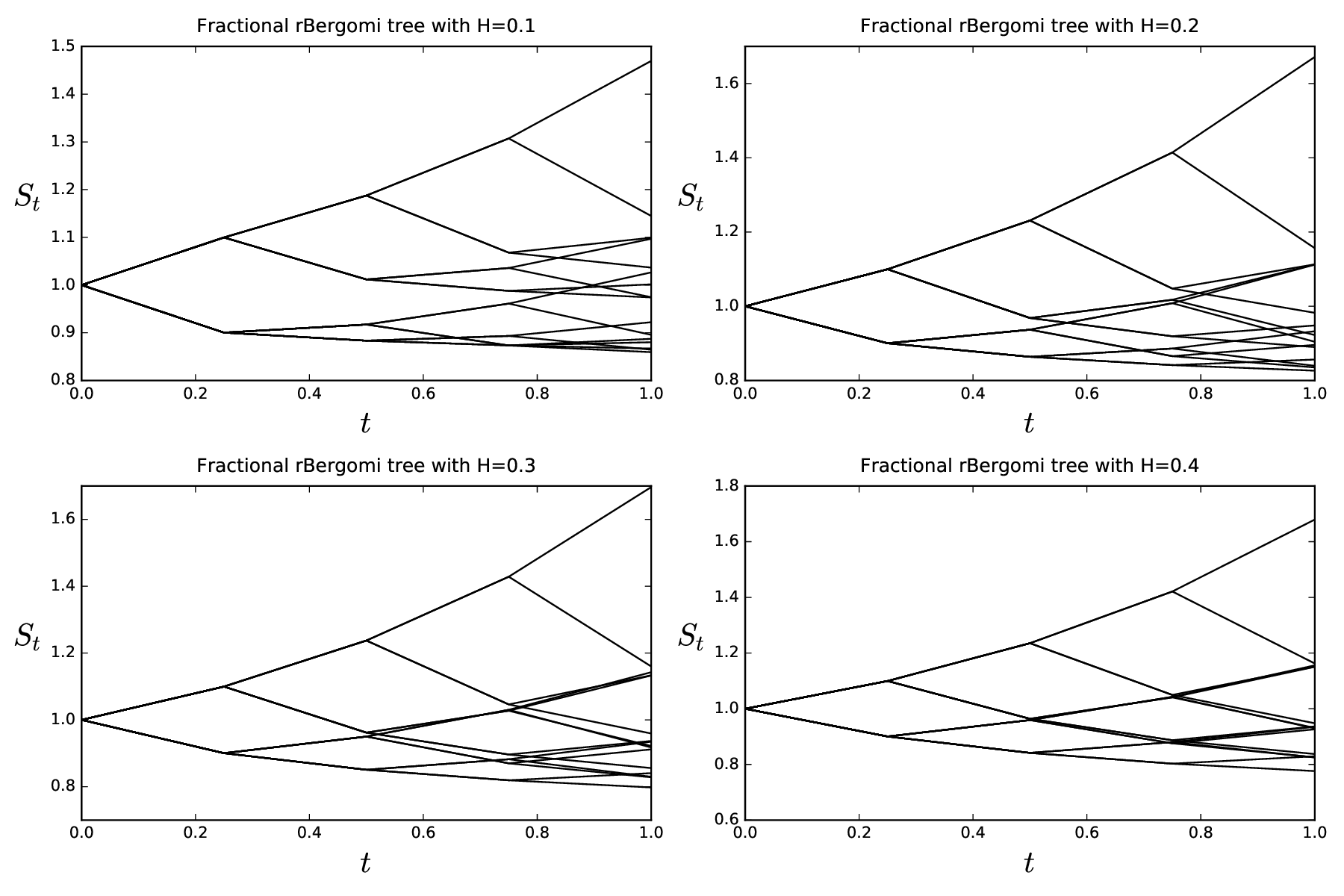}
\caption{rBergomi trees with $(\nu, \rho, \xi_0) = (1, 1, 0.04)$ 
and five time steps.}
\label{fig: rough_Bergomi_trees2}
\end{figure}

\subsubsection{European options}
Figure~\ref{fig: implied_trees} displays volatility smiles obtained using the tree scheme. 
Even though the time steps are not sufficient for small~$H$, the fit remarkably improves when $H\geq 0.15$, and always remains inside the $95\%$ confidence interval with respect to the Hybrid scheme. 
Moreover, the moment-matching approach from Section~\ref{sec:MomentMatching} 
shows a superior accuracy when $H\leq 0.1$, but is not sufficiently accurate. 
In Figure~\ref{fig: errors_trees} a detailed error analysis corroborates these observations:
the relative error is smaller than $3\%$ for $H\geq 0.15$. 

\begin{figure}[h!]
\centering
\includegraphics[scale=0.5]{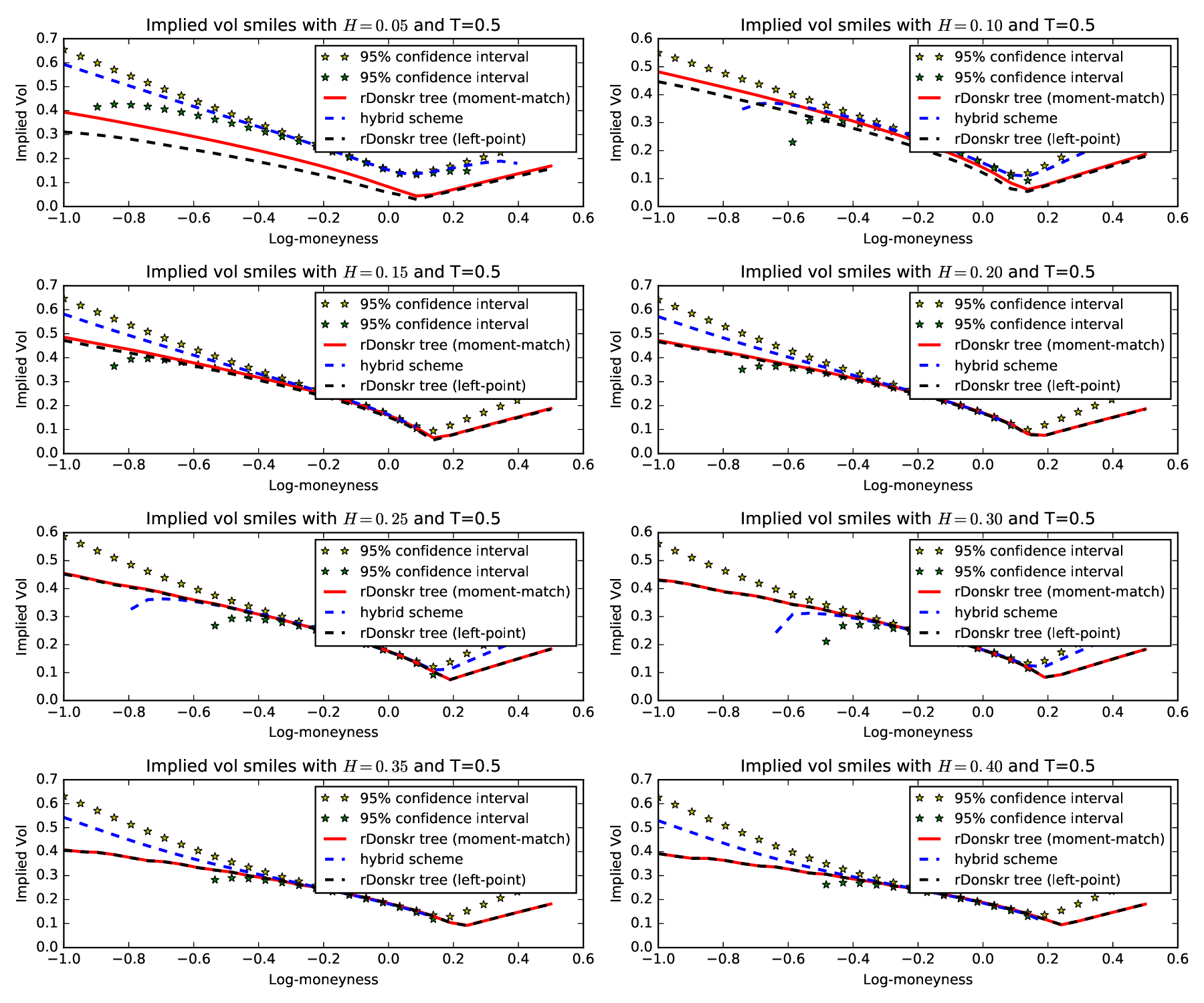}
\caption{rBergomi trees for different values of~$H$, $(\nu, \rho, \xi_0) = (1, -1, 0.04)$, $24$ time steps.}
\label{fig: implied_trees}
\end{figure}

\begin{figure}[h!]
\centering
\includegraphics[scale=0.5]{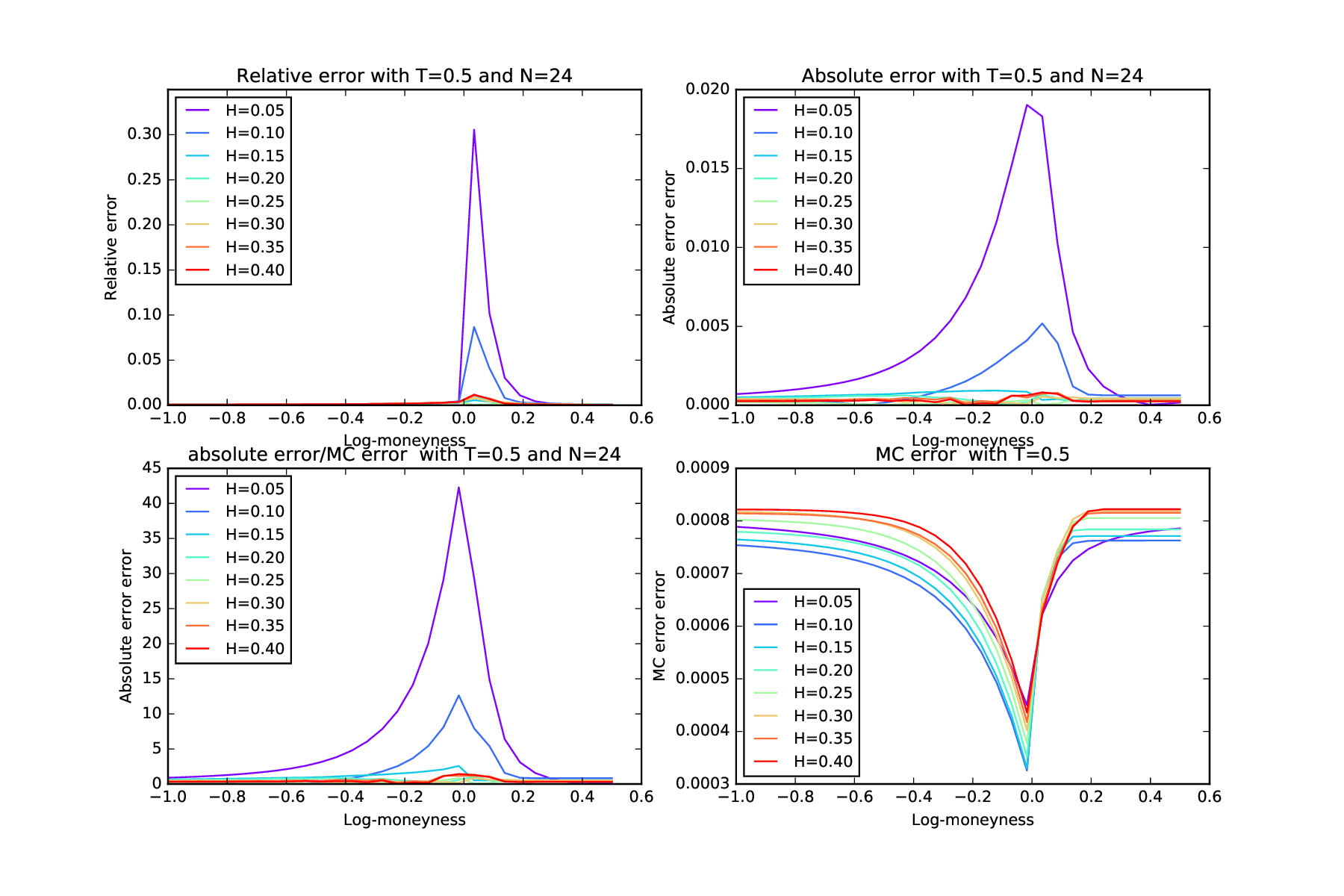}
\caption{Error analysis for the rDonsker moment-match tree for different values of~$H$, $(\nu, \rho, \xi_0) = (1, -1, 0.04)$ with~$24$ time steps.}
\label{fig: errors_trees}
\end{figure}

\subsubsection{American options}
In the context of American options, there is no benchmark to compare our result. 
However, the accurate results found in the previous section (at least for $H\geq0.15$) 
justify the use of trees to price American options. 
Figure~\ref{fig: american} shows the output of American and European Put prices 
with interest rates equal to $r=5\%$. 
Interestingly, the rougher the process (the smaller the~$H$), 
the larger the difference between in-the-money European and American options.

\begin{figure}[h!]
\centering
\includegraphics[scale=0.5]{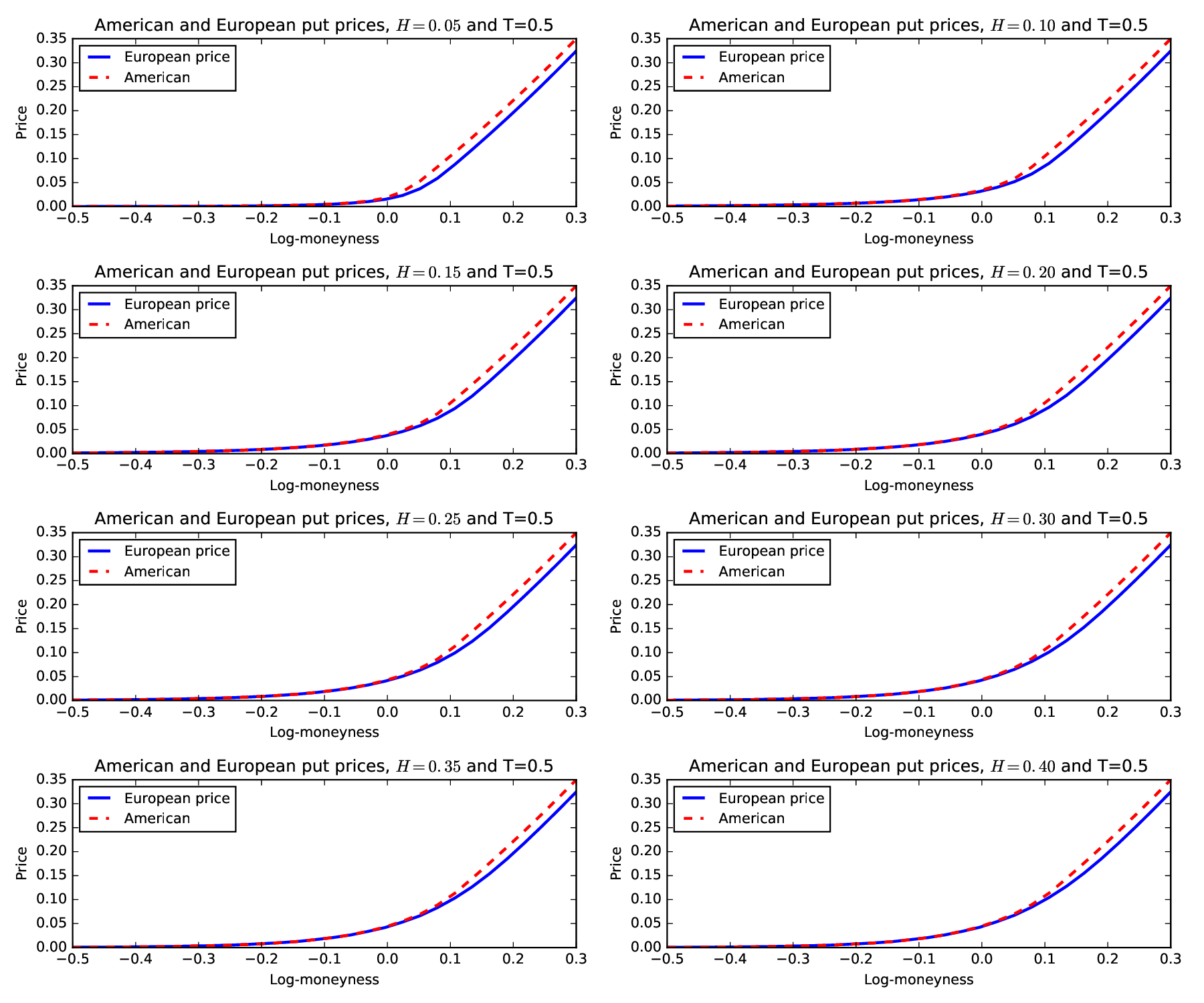}
\caption{American and European Put prices in the rough Bergomi model for different values of~$H$ 
and $(\nu, \rho, \xi_0) = (1, -1, 0.04)$ with~$26$ time steps.}
\label{fig: american}
\end{figure}

\newpage
\appendix
\section{Riemann-Liouville operators}\label{sec:Riemann-Liouville operators}
We review here fractional operators and their mapping properties. 
We follow closely the excellent monograph by Samko, Kilbas and Marichev~\cite{SKM93}, 
as well as some classical results by Hardy and Littlewood~\cite{HL28}. However, we introduce a modification in their definition, so that the condition $f(0)=0$ is not necessary as opposed to the original definition in Hardy and Littlewood~\cite{HL28}

\subsubsection{Riemann-Liouville fractional operators}
\begin{definition}
For $\lambda \in (0,1)$, $\alpha\in\mathfrak{R}^{\lambda}$
the left Riemann-Liouville fractional operator is defined on~$\Cc^{\lambda}(\TT)$ as
\begin{equation}\label{eq:RiemannLiouvilleoperator}
(I^\alpha f)(t) :=
\begin{cases}
\displaystyle\frac{1}{\Gamma(\alpha)}\int_0^t \frac{f(s)-f(0)}{(t-s)^{1-\alpha}}\D s, & \text{for }\alpha \in (0,1-\lambda),\\
\displaystyle\left(\frac{\D }{\D t}I^{1+\alpha}f\right)(t)
 = \frac{1}{\Gamma(1+\alpha)}\frac{\D}{\D t}\int_0^t(t-s)^\alpha (f(s)-f(0))\D s, & \text{for }\alpha \in (-\lambda, 0).
\end{cases}
\end{equation}
\end{definition}

\begin{theorem}\label{thm: integral map}
For any $f\in\Cc^\lambda(\TT)$, with $\lambda\in (0,1)$ and $\alpha\in\Rkl$,
$I^{\alpha}f \in \Cc^{\lambda+\alpha}(\TT)$.
In particular, there exists $C>0$ such that
$|(I^{\alpha}f)(t)| \leq C t^{\alpha+\lambda}$ for any $t \in \TT$.
\end{theorem}
\begin{proof}
We first consider $\alpha>0$, then we may easily represent
$$
(I^{\alpha}f)(t)
 = \frac{1}{\Gamma(\alpha)}\int_{0}^{t}\frac{f(u)-f(0)}{(t-u)^{1-\alpha}}\D u.
$$
Since $f\in\Cc^{\lambda}(\TT)$, we obtain
$\displaystyle
|(I^{\alpha}f)(t)|\leq \frac{|f|_\lambda}{\Gamma(\alpha)}\int_{0}^{t}\frac{u^\lambda \D u}{(t-u)^{1-\alpha}}$,
and hence
$$
|(I^{\alpha}f)(t)|
 \leq \frac{\Gamma(2+\lambda)|f|_\lambda}{(1+\lambda)\Gamma(\alpha+\lambda+1)}t^{\alpha+\lambda},
$$
which proves the estimate for $|I^{\alpha}f|$.
Next, we prove that $I^{\alpha}f\in\Cc^{\lambda+\alpha}(\TT)$. 
For this, introduce $\phi(t):=f(t)-f(0)$ and consider $t,t+h\in\TT$ with $h>0$,
\begin{align}\label{eq:I_alpha_estimate}
(I^{\alpha}f)(t+h) - (I^{\alpha}f)(t) & = \frac{1}{\Gamma(\alpha)}\left(\int_{-h}^{t}\frac{\phi(t-u)}{(u+h)^{1-\alpha}}\D u
 - \int_{0}^{t}\frac{\phi(t-u)}{u^{1-\alpha}}\D u\right)\\
 & = \frac{\phi(t)}{\Gamma(1+\alpha)}\left[(t+h)^\alpha-t^\alpha\right]+\frac{1}{\Gamma(\alpha)}\left(\int_{-h}^{0}\frac{\phi(t-u)-\phi(t)}{(u+h)^{1-\alpha}}\D u\right)\nonumber \\
 & + \frac{1}{\Gamma(\alpha)}\left(\int_{0}^{t}\left[(u+h)^{\alpha-1} - u^{\alpha-1}\right]\left[\phi(t-u)-\phi(t)\right]\D u\right)
 =: J_1+J_2+J_3.\nonumber
\end{align}
We first consider $J_1$. If $h>t$, then
$$
|J_1|\leq\frac{|f|_\lambda}{\Gamma(1+\alpha)}t^\lambda\left[(t+h)^\alpha-t^\alpha\right]
 \leq C h^{\lambda+\alpha}.
$$
On the other hand, when $0<h<t$, since $(1+u)^\alpha-1\leq \alpha u$ for $u>0$, then
$$
|J_1|
\leq \frac{|f|_\lambda}{\Gamma(1+\alpha)}t^{\lambda+\alpha}\left|\left(1+\frac{h}{t}\right)^{\alpha}-1\right|
\leq C h t^{\lambda+\alpha-1}\leq C h^{\lambda+\alpha}.
$$
For $J_2$, since $f\in\Cc^{\lambda}(\TT)$, we can write
$$
|J_2|\leq\frac{|f|_\lambda}{\Gamma(\alpha)}\int_{-h}^0\frac{|u|^\lambda}{(u+h)^{1-\alpha}}
\leq C h^{\lambda+\alpha}.
$$
Finally,
$$
|J_3|\leq\frac{|f|_\lambda}{\Gamma(\alpha)}\int_0^{t}u^\lambda[u^{\alpha-1}-(u+h)^{\alpha-1}]\D u
 = \frac{|f|_\lambda}{\Gamma(\alpha)}h^{\lambda+\alpha}\int_0^{t/h}u^\lambda[u^{\alpha-1}-(u+1)^{\alpha-1}]\D u.
$$
Hence, if $t\leq h$, then $|J_3|\leq C h^{\lambda+\alpha}$. 
Likewise, if $t>h$ and $\lambda+\alpha<1$, 
then $|J_3|\leq C h^{\lambda+\alpha}$ since
$$
\left|u^{\alpha-1}-(u+1)^{\alpha-1}\right|
 = u^{\alpha-1}\left[1-\left(1+\frac{1}{u}\right)^{\alpha-1}\right]\leq C u^{\alpha-2}.
$$
Thus, we have shown that $I^{\alpha}f$ satisfies the $(\lambda+\alpha)$-H\"older condition and belongs to $\Cc^{\lambda+\alpha}(\TT)$ in the case $\alpha >0$. The conclusion for $\alpha<0$ follows by taking $g(u):=u^\alpha$ in the proof of Proposition \ref{prop:GContinuous} in Appendix \ref{sect:app_C}.
\end{proof}

\begin{corollary}\label{cor: mapping Holder regularity}
For any $\lambda \in (0,1)$ and $\alpha\in\Rkl$,
$I^\alpha $ is a continuous operator from $\Cc^{\lambda}(\TT)$ to $\Cc^{\lambda+\alpha}(\TT)$.
\end{corollary}
\begin{proof}
It is clear that $I^\alpha $ is a linear operator. 
From Theorem~\ref{thm: integral map},
$\| I^\alpha f\|_{\alpha+\lambda}
 \leq  C_1\| f\|_\lambda\|(\cdot)^{\alpha+\lambda}\|_{\lambda+\alpha}
\leq C\| f\|_{\lambda}$,
since $|f|_\lambda\leq \| f\|_\lambda$.
Therefore $I^\alpha $ is also bounded and hence continuous.
\end{proof}

\section{Discrete convolution}\label{sec: convolution}
\begin{definition}\label{def: discrete convolution}
For $\mathrm{a},\mathrm{b}\in\RR^n$, 
the discrete convolution operator $\ast:\RR^n\times \RR^n\to\RR^n$ is defined as
$$
(\mathrm{a} * \mathrm{b})_{i} := \sum_{m=0}^{i} \mathrm{a}_{m} \mathrm{b}_{i-m},\quad i=0,\ldots,n-1.
$$
\end{definition}
When simulating $\Gg^\alpha W$ on the uniform partition~$\Tt$, the scheme reads
$$
(\Gg^\alpha W)^j(t_i)
 = \sum_{k=1}^{i}g(t_{i} - t_{k-1})\xi_{k}
 = \sum_{k=1}^{i}g(t_{k})\zeta_{j,k-i+1},
\qquad \text{for }i=1,\ldots,n,
$$
which has the form of the discrete convolution in Definition~\ref{def: discrete convolution}. 
Rewritten in matrix form, 
$$\begin{pmatrix}
g(t_{1}) & 0 & \cdots & 0\\
g(t_{2}) & g(t_{1}) & \cdots & 0\\
\vdots &\ddots &\ddots & 0\\
g(t_{n}) & g(t_{n-1}) & \cdots &   g(t_{1})
\end{pmatrix}
\begin{pmatrix}
\zeta_1\\
\vdots\\
\zeta_n
\end{pmatrix},
$$
it is clear that this operator yields a complexity of order~$\Oo(n^2)$, which can be improved drastically.
\begin{definition}\label{DFT}
The Discrete Fourier Transform (DFT) of a sequence $\cm:=\left(c_0,c_1,...,c_{n-1}\right)\in\mathbb{C}^n$ is given by 
$$
\widehat{f}(\cm)[j] := \sum_{k=0}^{n-1}c_k \exp\left(-\frac{2\I \pi j k}{n}\right),
\qquad\text{for }j=0,\ldots,n-1,
$$ 
and the Inverse DFT of~$\cm$ is given by 
$$
f(\cm)[k] := \frac{1}{n}\sum_{j=0}^{n-1}c_j \exp\left(\frac{2\I\pi j k}{n}\right),
\qquad\text{for }k=0,\ldots,n-1.
$$ 
\end{definition}
In general, both transforms require a computational effort of order $\Oo(n^2)$,
but the Fast Fourier Transform (FFT) algorithm by Cooley and Tukey~\cite{CT65} 
exploits the symmetry and periodicity of complex exponentials of the DFT 
and reduces the complexity of both transforms to~$\Oo(n\log n)$.
\begin{theorem}\label{thm: convolution theorem}
For $\am, \bm\in\RR^n$, the identity 
$(\am\ast \bm)=f\bigl(\widehat{f}(\am)\bullet\widehat{f}(\bm)\bigr)$ holds, 
with~$\bullet$ the pointwise multiplication.
\end{theorem}
This in particular implies that the complexity of the discrete convolution is reduced to~$\Oo(n\log n)$ by FFT.
\begin{algorithm}[FFT Discrete convolution for $\Bb$]\label{algo: convolution for rough vol}
On the equidistant grid~$\Tt$, 
\begin{enumerate}
\item draw a random matrix
$\{\zeta_{j,i}\}_{\substack{j=1,\ldots,M\\i=1,\ldots,n}}$ 
such that $\mathbb{V}(\zeta_{j,i})=1$;
\item\label{algo: step2} define the vectors
$\gr:=(g(t_i))_{i=1,\ldots,n}$
and
$\zeta_j:=(\zeta_{j,i})_{i=1,\ldots,n}$, for $j=1,\ldots,M$;
\item \label{algo: step3} using FFT, compute 
$\varphi_j:=\widehat{f}(\gr)\cdot\widehat{f}(\zeta_j)$, for $j=1,\ldots,M$;
\item \label{algo: step4}
simulate $M$ paths of~$(\Gg^\alpha W)$ using FFT, 
as $(\Gg^\alpha W)^j(\Tt)=\sqrt{\frac{T}{n}}f(\varphi_j)$ for $j=1,\ldots,M$.
\end{enumerate}
\end{algorithm}
In Step~\ref{algo: step2} we may replace the evaluation points~$\gr$ by any optimal evaluation point
 $\{g\left(t_i^*\right)\}_{i=1}^n$ as in~\eqref{eq: optimal eval}. 
Many packages offer a direct implementation of the discrete convolution such as \texttt{numpy.convolve} in~\texttt{Python}. 
The user then only needs to pass the arguments~$\gr$ and~$\xi_j$, 
and Steps~\ref{algo: step3} and~\ref{algo: step4} 
are computed automatically (using efficient FFT techniques) by the function. 
Although the FFT step is the heaviest computation on the simulation of rough volatility models, 
the actual time grid~$\Tt$ is not specially large ($n\ll 1000$). 
Hence, the fastest FFT for very large~$n$ is not essential, 
as the implementation is run on smaller time grids. 
In this aspect we find that \texttt{numpy.convolve} is a very competitive implementation.


\section{Proof of Proposition \ref{prop:GContinuous}}\label{sect:app_C}

In this section, we present a proof of Proposition \ref{prop:GContinuous}. We first consider the case $\alpha \in (-\lambda , 0)$ with $0<\lambda \leq 1$. Fix $f\in \mathcal{C}^\lambda(\mathbb{I})$ and $g \in \mathcal{L}^\alpha$. As our first step, we derive a useful representation akin to \cite[Equation (13.1)]{SKM93}, but for the operator $\mathcal{G}^\alpha$, which amounts to $\mathcal{G}^\alpha f(0)=0$ and
\begin{equation}\label{variant_of_13.1}
\mathcal{G}^\alpha f(t) =  (f(t)-f(0))g(t) - \int_0^t (f(t)-f(s))\frac{\mathrm{d}}{\mathrm{d}t}g(t-s)\mathrm{d}s,
\end{equation}
for $t\in (0,1]$ (note that $g(0)$ need not be defined, but our assumptions guarantee $\mathcal{G}^\alpha f(0)=0$). 
To show that~\eqref{variant_of_13.1} holds, we look at the difference quotients for the definition of~$\mathcal{G}^\alpha f$ in~\eqref{eq:GFO}. For any $t\in[0,1)$ and any small enough $h>0$, a bit of rewriting leads to the equality
\begin{align}\label{eq:difference}
 &\int_{0}^{t+h}(f(s)-f(0))g(t+h-s)\D s - \int_{0}^{t}(f(s)-f(0))g(t-s)\D s  \\  &\quad=  \int_0^t (f(s)-f(t))\bigl( g(t+h-s)-g(t-s)\bigr)\D s \nonumber \\
& \qquad + (f(t)-f(0))\biggl( \int_0^{t+h} g(t+h-s)\mathrm{d}s  - \int_0^t g(t-s) \mathrm{d}s \biggr)   + \int_t^{t+h} (f(s)-f(t))g(t+h-s)\mathrm{d}s.\nonumber
\end{align}
In the second term on the right-hand side of~\eqref{eq:difference}, a change of variables gives
\begin{equation*}
(f(t)-f(0))\biggl(\int_0^{t+h} g(t+h-s)\mathrm{d}s  - \int_0^t g(t-s) \mathrm{d}s \biggr) = (f(t)-f(0))\int_{-h}^0 g(t-r)\mathrm{d}r.
\end{equation*}
Looking at the third term on the right-hand side of~\eqref{eq:difference}, our assumptions yield
\[
\Bigl\vert \int_t^{t+h} (f(s)-f(t))g(t+h-s)\D s \Bigr\vert \leq  \int_t^{t+h} C_1 h^\lambda C_2 h^\alpha \D s \leq C_0 h^{1+\lambda + \alpha} = o(h),
\]
as $h$ tends to zero, since $\lambda +\alpha\in (0,1)$. As for the first term, we have
\begin{align*}
\left\vert  (f(t)-f(s))\frac{g(t+h-s)-g(t-s)}{h} \right \vert
&\leq \frac{ C_1 |t-s|^\lambda}{h}\int_{t-s}^{t-s+h} g^\prime(r) \D r
 \leq C_1 (t-s)^\lambda C_2 |t-s|^{\alpha -1} =  C_0 (t-s)^{\lambda+\alpha-1},
\end{align*}
for all $s\in (0,t)$ and $h>0$, where $\lambda + \alpha -1 \in (-1,0)$, so the right-hand side is in $L^1([0,t])$ and hence we can apply the dominated convergence theorem. Specifically, dividing by $h$ in~\eqref{eq:difference} and sending $h$ to zero, we obtain~\eqref{variant_of_13.1} in the limit, by using dominated convergence on the first term, Lebesgue's differentiation on the second term, and noting that the third term vanishes.

Having established~\eqref{variant_of_13.1}, we can now use it to obtain the desired H\"older estimates. We begin with the first term on the right-hand side of~\eqref{variant_of_13.1}. Let $\phi(t):=(f(t)-f(0))g(t)$, for $t\in \mathbb{I}$, where we note that $|(f(t)-f(0))g(t)|\leq C t^{\lambda + \alpha}$ with  $\lambda + \alpha \in (0,1)$, so $\phi(0)=0$ is well defined. 
Rewriting, and using the assumptions on~$f$ and~$g$, we get, for every $t\in \mathbb{I}$ and $h\in (0,1-t]$,
\begin{align}\label{eq:first_term_holder}
\vert \phi(t+h)-\phi(t) \vert &\leq |f(t)-f(0)| \int_{t}^{t+h}|g^\prime(r)|\mathrm{d}r + |g(t+h)||f(t+h)-f(t)| \\
& \leq  C \vert f \vert_\lambda t^{\lambda + \alpha} (t+h)^\alpha \bigl( (t+h)^{-\alpha} - t^{-\alpha} \bigr) +   C \vert f \vert_\lambda h^{\lambda + \alpha} \leq C^\prime  \vert f \vert_\lambda h^{\lambda +\alpha},\nonumber
\end{align}
where the last inequality follows by elementary considerations, as in the arguments on \cite[Chapter~1, Page~15]{Mus72}. 
The case $h\in [-t,0]$ is analogous. 

For the second term on the right-hand side of~\eqref{variant_of_13.1}, we can follow a procedure similar to the proof of \cite[Lemma 13.1]{SKM93}. Defining
\[
\varphi(t):=\int_0^t (f(t)-f(s))\frac{\mathrm{d}}{\mathrm{d}t}g(t-s)\mathrm{d}s = \int_0^t (f(t)-f(t-r))\frac{\mathrm{d}}{\mathrm{d}r}g(r)\mathrm{d}r  ,
\]
and rewriting things, for any $t\in\mathbb{I}$ and $h\in[-t,1-t]$, we arrive at
\begin{align}\label{eq:second_term_holder}
	\varphi(t+h)-\varphi(t) &= \int_0^t (f(t) - f(t-r))\bigl(g^\prime(r+h) - g^\prime(r)\bigr)\mathrm{d}r - \int_0^t  (f(t) - f(t-r))g^\prime(r+h)\mathrm{d}r \\
	& \quad\qquad+ \int_{-h}^t (f(t+h)-f(t-u)) g^\prime(u+h)\mathrm{d}u \nonumber\\
	&= \int_0^t (f(t) - f(t-r))\bigl(g^\prime(r+h) - g^\prime(r)\bigr)\mathrm{d}r - \int_0^t  (f(t+h) - f(t))g^\prime(r+h)\mathrm{d}r \nonumber\\
	& \quad\qquad+ \int_{-h}^0 (f(t+h)-f(t-u)) g^\prime(u+h)\mathrm{d}u := I_1 + I_2 + I_3. \nonumber
\end{align}
Without loss of generality, we assume $h>0$. For the first integral, a change of variables gives
\begin{align*}
	|I_1| &\leq \vert f \vert_\lambda \int_0^t r^\lambda \int_{r}^{r+h} |g^{\prime \prime} (u)|\mathrm{d}u \mathrm{d}r \leq C \vert f \vert_\lambda \int_0^t r^\lambda \bigl( r^{\alpha-1} - (r+h)^{\alpha-1}  \bigr) \mathrm{d}r \\
	& =  C \vert f \vert_\lambda h^{\lambda+\alpha -1} \int_0^t \bigl(\frac{r}{h}\bigr)^\lambda \Bigl( \bigl(\frac{r}{h}\bigr)^{\alpha-1} - \bigl(\frac{r}{h}+1\bigr)^{\alpha-1}  \Bigr) \mathrm{d}r = C \vert f \vert_\lambda h^{\lambda+\alpha } \int_0^{t/h} u^{\alpha+\lambda-1} \Bigl(1 -\bigl( 1+\frac{1}{u}\bigr)^{\alpha -1} \Bigr) \mathrm{d}u \\
	& \leq C \vert f \vert_\lambda h^{\lambda+\alpha } \Bigl( \int_0^{1} u^{\lambda+\alpha -1} \mathrm{d}u + (1-\alpha )\int_1^{\infty} u^{\lambda+\alpha -2} \mathrm{d}u \Bigr),
\end{align*}
where $\lambda+\alpha\in(0,1)$, so the final two terms on the right-hand side are finite. In the final line, we have used that the mapping $y\mapsto - (1+y)^{\alpha -1}+(\alpha-1)y$ is concave with a maximum value of $-1$ at $y=0$.

As regards the two remaining integrals $I_2$ and $I_3$, we see immediately that
\begin{align*}
	&|I_2| \leq C \vert f \vert_\lambda h^{\lambda} \int_0^\infty (r+h)^{\alpha -1} \mathrm{d}r= \frac{C}{|\alpha|} \vert f \vert_\lambda h^{\lambda+\alpha},\quad \text{and} \\
	&|I_3| \leq C \vert f \vert_\lambda	\int_{-h}^0 (u+h)^{\lambda+\alpha - 1} \mathrm{d}u = \frac{C}{|\lambda +\alpha|}  \vert f \vert_\lambda h^{\lambda + \alpha}.
\end{align*}
By linearity, the desired continuity of the operator 
$\mathcal{G}^\alpha:\mathcal{C}^\lambda(\mathbb{I})\rightarrow \mathcal{C}^{\lambda+\alpha}(\mathbb{I})$, for $\alpha\in(-\lambda, 0)$, now follows from~\eqref{variant_of_13.1}, ~\eqref{eq:first_term_holder}, and the three above estimates for~\eqref{eq:second_term_holder}.

It remains to consider $\alpha \in (0,1-\lambda) $. As before, recall $0<\lambda \leq 1$, and fix $f\in \mathcal{C}^\lambda(\mathbb{I})$ along with $g\in \mathcal{L}^\alpha$. Unlike above, $s\mapsto \frac{\mathrm{d}}{\mathrm{d}t}g(t-s)$ is now integrable on $\mathbb{I}$ which makes things go through more easily:~in particular, we can work directly with the definition of $\mathcal{G}^\alpha f$ in~\eqref{eq:GFO}, applying arguments analogous to~\eqref{eq:second_term_holder}. The case $g(u)=u^\alpha$ is already covered by the proof of Theorem \ref{thm: integral map}. For a general $g \in \mathcal{G}^\alpha$, we can  retrace those same steps, except that, in~\eqref{eq:I_alpha_estimate} and the subsequent estimates for $J_1$, $J_2$, and $J_3$, we must now invoke our control on $g$ and its derivatives (similarly to how we did it above for~\eqref{eq:second_term_holder} and the subsequent estimates of $I_1$, $I_2$, and $I_3$). This completes the proof of Proposition \ref{prop:GContinuous}.

\newpage


\end{document}